\RequirePackage{fix-cm}
\RequirePackage{amsmath}
\documentclass[smallextended]{svjour3}       
\smartqed  
\usepackage{graphicx}
\usepackage{amsfonts}
\usepackage{amssymb}
\usepackage{latexsym}
\usepackage{mathtools}
\mathtoolsset{showonlyrefs}
\usepackage{listings}
\usepackage{graphicx}
\usepackage{color}
\usepackage{xcolor}
\usepackage{fullpage}
\usepackage{hyperref}
\usepackage{braket}

\newcommand{\wt}{\widetilde}
\newcommand{\wh}{\widehat}

\newcommand{\serie}[1]{\{#1_{n}\}_n}

\def\C{\mathbb{C}}
\def\R{\mathbb{R}}

\def\trace{\mathrm{tr}}
\def\rank{\mathrm{rank}}

\newcommand{\lowbiglquote}[1][18]{%
	\setbox0=\hbox{\fontsize{#1}{0}\selectfont``}%
	\setlength{\dimen0}{\ht0 - \heightof{A}}%
	\noindent\llap{\smash{\lower\dimen0\box0 }}}

\newcommand{\lowbigrquote}[1][18]{%
	\setbox0=\hbox{\fontsize{#1}{0}\selectfont''}%
	\setlength{\dimen0}{\ht0 - \heightof{A}}%
	\unskip\rlap{\smash{\lower\dimen0\box0 }}}

\def\Prob{\mathbb{P}}

\newcommand{\Expect}{\operatorname{\mathbb{E}}}

\newcommand{\normal}{\mathcal{N}}

\newcommand{\vf}{\varphi}
\newcommand{\ve}{\varepsilon}
\newcommand{\sqsize}[1]{{#1\times#1}}
\newcommand{\f}{\mathbb}

\spnewtheorem{theorem}{Theorem}[section]{\bfseries}{\itshape}

\spnewtheorem{lemma}[theorem]{Lemma}{\bfseries}{\itshape}
\spnewtheorem{definition}[theorem]{Definition}{\bfseries}{\itshape}
\spnewtheorem{remark}[theorem]{Remark}{\bfseries}{\upshape}
\spnewtheorem{example}[theorem]{Example}{\bfseries}{\upshape}
\spnewtheorem{corollary}[theorem]{Corollary}{\bfseries}{\itshape}
\spnewtheorem{proposition}[theorem]{Proposition}{\bfseries}{\itshape}
\spnewtheorem{algorithm}[theorem]{Algorithm}{\bfseries}{\upshape}

\numberwithin{theorem}{section}

\begin{document}
	
	\title{The limit empirical spectral distribution of Gaussian monic complex matrix polynomials}
	
	\titlerunning{The empirical spectral distribution of Gaussian monic complex matrix polynomials}        
	
	\author{Giovanni Barbarino        \and
		Vanni Noferini 
	}
	
	\institute{G. Barbarino \at
		Department of Mathematics and Systems Analysis, Aalto University, Finland.
		\email{giovanni.barbarino@aalto.fi}    
		\and
		V. Noferini \at
		Department of Mathematics and Systems Analysis, Aalto University, Finland.
		\email{vanni.noferini@aalto.fi}  
	}
	
	\date{}	
	
	\maketitle

	\begin{abstract}
		We define the empirical spectral distribution (ESD) of a random matrix polynomial with invertible leading coefficient, and we study it for complex $n \times n$ Gaussian monic matrix polynomials of degree $k$. We obtain exact formulae for the almost sure limit of the ESD in two distinct scenarios: (1)  $n \rightarrow \infty$ with $k$ constant and (2) $k \rightarrow \infty$ with $n$ constant. The main tool for our approach is the replacement principle by Tao, Vu and Krishnapur. Along the way, we also develop some auxiliary results of potential independent interest: we slightly extend a result by B\"{u}rgisser and Cucker on the tail bound for the norm of the pseudoinverse of a non-zero mean matrix, and we obtain several estimates on the singular values of certain structured random matrices.
		
	\end{abstract}

	\section{Introduction}\label{sec:introduction}
	
	Given $n \times n$ matrices $C_0, C_1, \dots, C_k \in \C^{n \times n}$, consider the square matrix polynomial of degree $k$
	\begin{equation}\label{eq:matpoly}
	P(x) = \sum_{j=0}^k C_j x^j;
	\end{equation}
	a finite eigenvalue of $P(x)$ is then defined \cite{Dopico2018,LN20} as a number $\lambda \in \C$ such that
	\[  \rank _\C P(\lambda) <  \rank_{\C(x)} P(x). \]
	The polynomial eigenvalue problem (PEP) is to find all such eigenvalues \cite{AT12,DLPVD18,Dopico2018,GLRBook,GT17,LN20,NP15,TM01}, possibly (and depending on the application) together with other objects -- such as eigenspaces, infinite eigenvalues, minimal indices and minimal bases -- whose precise definition is not relevant for this paper. Under the generic assumption that $\det P(x) \not \equiv 0$, the finite eigenvalues of $P(x)$ are the roots of its determinant. Polynomial eigenvalue problems are common in several areas of applied and computational mathematics; their applications include acoustics, control theory, fluid mechanics, and structural engineering \cite{GLRBook,GT17,TM01}.
	
	Clearly, two very classical mathematical problems arise as special cases of polynomial eigenvalue problems: finding the roots of a scalar polynomial corresponds to $n=1$, while finding the eigenvalues of a matrix corresponds to $k=1$ and $C_1=I_n$. When randomness enters the game, these two extremes are well understood. It is known that, when the polynomial coefficients are i.i.d. normally distributed random variables and in the limit $k \rightarrow \infty$, then the roots of scalar polynomials are uniformly distributed on the unit circle.  Similarly, classical results in random matrix theory state that, when the entries of an $n \times n$ matrix are i.i.d. normally distributed random variables with mean $0$ and variance $n^{-1}$, and in the limit $n \rightarrow \infty$, then the eigenvalues are uniformly distributed on the unit disc. Moreover, the phenomenon of universality is well known: there exist works that, under relatively mild assumptions, extend these results to several other distributions of coefficients or entries.
	
	To our knowledge, nothing was so far explicitly known about the eigenvalues of random matrix polynomials, except for the two extremal cases above described. In this paper, we fill this gap by computing the empirical eigenvalue distribution of  monic ($C_k=I)$ square matrix polynomials of size $n$ and degree $k$, with all but the leading coefficient being i.i.d. complex Gaussian random matrices,  in two different limits: when $n \rightarrow \infty$ with $k$ constant, and when $k \rightarrow \infty$ with $n$ constant. Moreover, our results can equivalently be interpreted as results on the empirical eigenvalue distribution of certain structured random matrices: indeed, given a monic matrix polynomials $P(x)$, a \emph{linearization} of $P(x)$ is a matrix whose eigenvalues (as well as their geometric and algebraic multiplicities) coincide with those of $P(x)$. In the numerical linear algebra literature, numerous constructions of linearizations are known: see, e.g., \cite{DLPVD18,GLRBook,NP15} and the references therein. In particular, the prototype of all linearizations is the so-called companion matrix, which plays a central role in this paper.
	
	In previous research on matrix polynomials, probability theory was used in the context of analysing the condition number of PEPs. Namely, in \cite{AB19} Armentano and Beltr\'{a}n  computed the average eigenvalue condition number for Gaussian random complex matrix polynomials; and in \cite{BK19}, Beltr\'{a}n and Kozhasov extended the analysis to the case of real Gaussian matrix polynomials. In \cite{LN20}, Lotz and Noferini went beyond the classical idea of condition by imposing a uniform probability distribution on the sphere for perturbations of a fixed singular matrix polynomial. However, we are not aware of any previous work where the exact distribution of the eigenvalues of a random matrix polynomial is obtained. In addition to being interesting \emph{per se}, our results can potentially be valuable to numerical analysts in the context of testing numerical methods for the solution of the PEP. Indeed, although randomly generated problems are expected not to be very challenging from the numerical point of view (by the results in \cite{AB19,BK19}), it is common practice to use them as benchmark for minimal performance requirements;  in published research papers on this subject, tests on random input are in fact often included among the numerical experiments. The analytic knowledge of the limit eigenvalue distributions that we obtain in this article can help to predict the behaviour of randomly generated problems: when scrutinizing a novel algorithm, if the numerically computed eigenvalues should significantly deviate from the expectations then this fact can raise legitimate suspicions on the accuracy of the computations.

	The structure of the paper is as follows. In Section \ref{sec:preliminaries} we review some necessary background material on linear algebra, matrix polynomial theory, probability theory, and random matrix theory. Moreover, we define the empirical spectral distribution of a random matrix polynomial with invertible leading coefficient. In Section \ref{sec:n} we obtain our first main result: the almost sure limit, for $n \rightarrow \infty$, of the empirical spectral distribution of a random $n \times n$ monic complex Gaussian matrix polynomial of degree $k$. In Section \ref{sec:k}, our second main result is discussed: the almost sure limit, for $k \rightarrow \infty$, of the empirical spectral distribution of a random $n \times n$ monic complex Gaussian matrix polynomial of degree $k$. In Section \ref{sec:concl} we draw some conclusions and propose new lines of research. To keep the main part of the paper as easily readable as possible, the proof of some technical lemmata, needed in Sections \ref{sec:n} and \ref{sec:k}, is postponed to Appendix \ref{sec:app}; however, we believe that some of those results could have independent interest. In particular, we slightly improve known results on the tail bounds for pseudoinverses of random matrices with nonzero mean, and we study the extremal singular values of certain structured random matrices.
	
	\section{Mathematical background}\label{sec:preliminaries}

	\subsection{Linear algebra}

	Given an $m \times n$ complex matrix $X$, we denote it singular values by $\sigma_1(X) \geq \dots \geq \sigma_{\min}(X) \geq 0$, having introduced the shorthand $\sigma_{\min}(X) :=\sigma_{\min(m,n)}(X) $. The spectral norm of $X$ is denoted by $\| X \| :=\sigma_1(X)$, while the Frobenius norm of $X$ is
	\[ \| X \|_F = \left( \sum_{i=1}^m \sum_{j=1}^n |X_{ij}|^2 \right)^{1/2} = \sqrt{\trace(X^* X)} = \sqrt{\sigma_1(X)^2 + \dots + \sigma_{\min}(X)^2} . \]
	Recall that any $X$ admits a singular value decomposition $U\Sigma V$ with $U\in \C^{m\times m}$, $V\in \C^{n\times n}$ unitary matrices and $\Sigma\in \R^{m\times n}$ diagonal real matrix whose diagonal elements are the singular values $\sigma_i(X)\ge 0$. 
	The Moore-Penrose pseudoinverse of $X$ is the matrix $X^\dagger=V^*\Sigma^\dagger U^*$, where $\Sigma^\dagger\in \R^{n\times m}$ is a diagonal real matrix whose diagonal entries are $\Sigma_{i,i}^\dagger =\Sigma_{i,i}^{-1} = \sigma_i(X)^{-1}$ if $\sigma_i(X)>0$  and zero otherwise. 
	Note that, if $X$ has full rank,  then $\|X^{\dagger}\| = 1/\sigma_{\min}(X)$. We also use the induced $1$ and $\infty$ matrix norms, defined respectively as
	\[
	\|X\|_1 = \max_{1\le j\le n} \sum_{i=1}^m |X_{ij}|,
	\qquad
	\|X\|_\infty = \max_{1\le i\le m} \sum_{j=1}^n |X_{ij}|.
	\]
	Since $\C^{m \times n}$ is finite dimensional, the various norms mentioned above are of course equivalent to each other, and the following relations will be useful to us:
	\begin{equation}
	\|X\|\le \|X\|_F,\qquad	\|X\|\le \sqrt{ \|X\|_\infty\|X\|_1   }.
	\end{equation}

	\noindent
	An interlacing result for the singular values arises when we consider low rank perturbation of matrices.
	\begin{theorem}[Interlacing Singular Values for Low-Rank Perturbations \cite{THO76}]\label{thm:Interlacing2}
		Let $A$ and $E$ be  $n\times n$ matrices, where $E$ has rank at most $k$. If $B = A+E$ and the singular values of $A$ and $B$ are, respectively, 
		\[
		\alpha_1\ge \alpha_2 \ge\dots\ge \alpha_{n}, \qquad
		\beta_1\ge \beta_2 \ge\dots\ge \beta_{n},
		\]
		then
		\begin{align*}
		&\alpha_i\ge \beta_{i+k}, & i=1,2,\dots,n-k,\\
		&\beta_i\ge \alpha_{i+k}, & i=1,2,\dots,n-k.
		\end{align*}
	\end{theorem}
	
	\noindent
	If the norm of the perturbation, as opposed to its rank, is to be used to estimate the singular values, then we can appeal to the following result attributed to Mirsky, which is a corollary of the minimax principle for singular values.
	\begin{theorem}[Perturbation Theorem \cite{MIR60}]\label{thm:Perturbation}
		Given two $\sqsize{ n}$ matrices $A,B$,  with singular values, respectively, 
		\[
		\alpha_1\ge \alpha_2 \ge\dots\ge \alpha_n, \qquad
		\beta_1\ge \beta_2 \ge\dots\ge \beta_n,
		\]
		then 
		\[
		|\alpha_i-\beta_i|\le \|A-B\|, \qquad i=1,2,\dots,n.
		\]
	\end{theorem}

	\subsection{Matrix polynomial theory}
	
	Let $P(x)$ be the matrix polynomial defined in \eqref{eq:matpoly}. We give here a brief overview of those aspects in the spectral theory of square complex matrix polynomials that are relevant to this paper. More detailed discussions can be found, e.g., in~\cite{AT12,Dopico2018,GLRBook,LN20} and the references therein.
	As mentioned in the introduction, an element $\lambda \in \C$ is said to be a finite eigenvalue of $P(x)$ if
	\begin{equation*}
	\rank_{\C}(P(\lambda)) < \rank_{\C(x)}(P(x)) =: r,
	\end{equation*}
	where $\C(x)$ is the field of fractions of $\C[x]$, that is, the field of rational functions with coefficients in $\C$.

	If the leading coefficient $C_k$ of the matrix polynomial $P(x)$ in \eqref{eq:matpoly} is invertible, then $P(x)$ has $kn$ finite eigenvalues. Under this assumption, one can define the \emph{companion matrix of $P(x)$} as (see e.g.  \cite{AT12})
	\begin{equation}\label{eq:companionmatrix}
	M=\begin{bmatrix}
	-C_k^{-1}C_{k-1} & -C_k^{-1}C_{k-2} & \dots & -C_k^{-1}C_{1} & -C_k^{-1}C_{0}\\
	I_n & 0 & 0 & \dots & 0\\
	0& I_n & 0 & \dots & 0\\
	\vdots & \ddots & \ddots & \ddots & \vdots  \\
	0 & \dots & 0 & I_n & 0
	\end{bmatrix} \in \C^{kn \times kn},
	\end{equation}
	where $I_n$ and $0$ are, respectively, the $n\times n$ identity and zero matrices. It is well known that the eigenvalues of $M$, defined in the classical sense, coincide with the finite eigenvalues of $P(x)$. As a consequence, under the assumption that $C_k$ is invertible,  studying the finite eigenvalues of  $P(x)$ is equivalent to studying the eigenvalues of the structured matrix $M$. Observe that, if $P(x)$ is monic, then $C_k = I$ so that the assumption is automatically satisfied.
	
	We can identify, say via an arbitrary but fixed rearrangement of the real and imaginary parts of the entries of each coefficient, the (real) vector space of $n \times n$ complex matrix polynomials of degree up to $k$ with $\R^{2(k+1)n^2}$. In this setting, let $\mathcal{S} \subset \R^{2(k+1)n^2}$ correspond to the subset of matrix polynomials that are regular and have $kn$ distinct finite eigenvalues. We conclude this subsection by observing that $\mathcal{S}$ is a nonempty Zariski open set, and hence, its complement has Lebesgue measure zero: in this sense, being regular with $kn$ distinct finite eigenvalues is a generic property of matrix polynomials.

	\subsection{Random Matrix Theory}\label{sec:Random_Matrix_Theory}

	Often, within our probabilistic arguments it will be crucial to consider matrices that have some deterministic entries and some other entries corresponding to (complex) random variables, which in turn can be seen as pairs of real random variables. We implicitly identify those matrices with a vector in $\R^N$, $N$ being the number of real random variables involved, and equipping $\R^N$ with an appropriate probability measure. In this context, we will often invoke, without explicit justification, the well known fact that events that happen in (subsets of) proper Zariski closed sets of $\R^N$ have probability zero: for example, we may claim that a certain square random matrix is almost surely invertible. Recalling that any proper Zariski closed set has Lebesgue measure zero, it follows immediately that the claimed property is true for any absolutely continuous probability measure (as are all the ones we discuss in this paper). The verification that, in all the instances where we make such a claim, the corresponding algebraic set is indeed contained in a proper algebraic set is a straightforward exercise in linear algebra, and we therefore omit the details.

	\subsubsection{Empirical spectral distributions}

	Given  a deterministic matrix $A\in M_{m}(\C)$ with eigenvalues $\lambda_1(A),\dots,\lambda_m(A)$ we say that its empirical spectral distribution (ESD) is the atomic measure
	\[
	\mu_A = \frac{1}{m} \sum_{i=1}^{m} \delta_{\lambda_i(A)},
	\]
	where the eigenvalues are considered with their respective algebraic multiplicities. A random matrix $A_n$ can be seen as a random variable with values in the appropriate space of matrices, that will usually be $M_{nk}(\f C)$, where $n$ and $k$ are fixed parameters. 
	We can extend the concept of ESDs to random matrices as follows.
	\begin{definition}\label{def:ESD}
		Given a random matrix $A$, its \textbf{empirical spectral distribution (ESD)} is a random variable with values in the space of probabilities on $\f C$, defined as
		\[
		\mu_A(\omega) := \mu_{A(\omega)} = \frac{1}{m} \sum_{i=1}^{m} \delta_{\lambda_i(A(\omega))}.
		\]	
	\end{definition}
	The space of probabilities on $\f C$ is a measurable subset of $\mathcal M^b(\f C)$, the space of signed measure of $\f C$ with bounded total variation, that is a Hausdorff space when equipped with the \textit{vague} (or \textit{weak$-^*$}) convergence of measures. 
	
	We will study the spectral distribution for some families of random matrices $\serie A$, and find that in our cases the sequence $\{\mu_{A_n}\}_n$ always converges almost surely (a.s.) to a constant random variable, that can be identified with a probability measure $\mu\in \f P(\f C)$. In this case, we simply write
	\[
	\mu_{A_n}\xrightarrow{a.s.} \mu.
	\] 
	The measure $\mu$ will thus be our candidate for the asymptotic
	spectral distribution of the family $\serie A$. 
	
	Finally, having let us consider a random matrix polynomial $P(x;\omega)$ of shape $n \times n$ and degree $k$, under the assumption that, for all $\omega \in \Omega$, $P(x;\omega)$ has invertible leading coefficient. This implies, in particular, that $P(x;\omega)$ has $kn$ finite eigenvalues, that we denote by $\lambda_1(P(x;\omega)),\dots,\lambda_{kn}(P(x;\omega))$.
	\begin{definition}\label{def:ESDP}
		Let $P(x;\omega)$ be a random matrix polynomial of size $n$ and degree $k$, such that its leading coefficient is invertible for all $\omega \in \Omega$. Its \textbf{empirical spectral distribution (ESD)}  is a random variable with values in the space of probabilities on $\C$, defined as
		\[
		\mu_P(\omega) := \mu_{P(x;\omega)} =\frac{1}{kn} \sum_{i=1}^{kn} \delta_{\lambda_i(P(x;\omega))}.
		\]	
	\end{definition}
	It is immediate by Definitions \ref{def:ESD} and \ref{def:ESDP} that the ESD of a random matrix polynomial coincides with the ESD of its (random) companion matrix \eqref{eq:companionmatrix}. Indeed, in this paper we will strongly rely on its equivalence.

	\subsubsection{The replacement principle and the circle law}
	
	Central to our arguments to derive the empirical spectral distributions is the so-called replacement principle: a tool in random matrix theory developed by Tao, Vu and Krishnapur. We recall it below.
	
	\begin{theorem}[Replacement Principle \cite{TVK10}]\label{thm:replacement}
		Let $A_m, B_m$ be two $m \times m$ random matrices. Assume that
		\begin{enumerate}
			\item The quantity $ \frac{1}{m^2} \left( \| A_m \|_F^2 + \|B_m\|^2_F \right) $ is bounded a.s.;
			\item For a.e.\ $z \in \C$,
			\[  \frac1m    \log \left|  \frac{\det( m^{-1/2}A_m  - z I  ) }{\det( m^{-1/2}B_m  - z I )}     \right|   \xrightarrow{a.s} 0. \]
			Then, $\mu_{\frac 1{\sqrt m}A_m} - \mu_{\frac 1{\sqrt m}B_m} \xrightarrow{a.s.} 0$.
		\end{enumerate}
	\end{theorem}
	
	\begin{remark}
		The random variable $\mu_{\frac 1{\sqrt m}A_m} - \mu_{\frac 1{\sqrt m}B_m}$ takes values in the space of signed measures on $\f C$ with  total variation bounded by $2$. 
	\end{remark}
	
	Thanks to the replacement principle, we will be able to generalize a well-known result on random Gaussian matrices to the case of monic Gaussian matrix polynomials.

	\begin{theorem}[Strong Circle Law \cite{MehtaBook}]\label{thm:Circle}
		Let $A_m$ 
		be the $\sqsize{m}$  random matrix whose entries
		are iid Gaussian random variables with mean 0 and variance 1. Then
		the ESDs of 
		$\frac 1{\sqrt m}  A_m$
		converges almost surely to the uniform distribution on the unit disc.
	\end{theorem}

	\section{Empirical spectral distribution for $n \times n$ monic complex Gaussian matrix polynomials of degree $k$, in the limit $n \rightarrow \infty$}\label{sec:n}
	Let $X$ be a complex random variable, normally distributed with mean $0$ and variance $1$. We consider the $n \times n$ monic matrix polynomial of degree $k \geq 2$
	\begin{equation}\label{eq:Px}
	P_n(x) = I_n x^k + \sum_{j=0}^{k-1} C_j x^j,
	\end{equation}
	where, for $j=0,\dots,k-1$ every coefficient $C_j$ is an $n \times n$ random matrix whose entries are i.i.d. copies of $X$. Note that each $C_j$ depends on $j$ and on $n$, but we omit the dependence on $n$ in the notation. It is intended moreover that all $C_j$ are independent of each other for varying $j$ and $n$. 
	
	The finite eigenvalues of $P_n(x)$ coincide with the eigenvalues of its companion matrix: in particular, substituting $C_k=I_n$ in \eqref{eq:companionmatrix}, we obtain
	\begin{equation}\label{eq:C}
	M:=\begin{bmatrix}
	-C_{k-1} & \dots & -C_1 & -C_0\\
	I_n & & & \\
	& \ddots & & \\
	& & I_n & 
	\end{bmatrix} =: Z + E_1 C^T
	\end{equation}
	where $E_1^T = \begin{bmatrix}
	I_n & 0 & \dots & 0
	\end{bmatrix}$ and $C^T = -\begin{bmatrix}
	C_{k-1} & \dots & C_1 & C_0
	\end{bmatrix}$. Note that the spectrum of the matrix $E_1 C^T$ consists of the eigenvalues of the random matrix $-C_{k-1}$, with the addition of the eigenvalue $0$, which appears with algebraic multiplicity $n(k-1)$. As $C_{k-1}$ is a Gaussian random matrix, the almost sure limit ESD of $n^{-1/2} C_{k-1}$ follows the circular law (\autoref{thm:Circle}), i.e., is distributed with the uniform measure on the unit disc. Hence, the ESD of $n^{-1/2} E_1 C^T$ converges almost surely, in the limit $n \rightarrow \infty$, to $\frac{k-1}{k} {\bf 1}_0 +  \frac1k {\bf 1}_D$, where ${\bf 1}_0$, ${\bf 1}_D$ denote the uniform probability measures on, respectively, the set $\{ 0 \}$ and the unit disc. Since $n^{-1/2}M$ is a perturbation of $n^{-1/2} E_1 C^T$, one can expect that the almost sure limit ESD of $n^{-1/2}M$, and thus 
	the almost sure limit ESD of $P_n(n^{1/2}x)$ coincides with the limit ESD for $n^{-1/2} E_1 C^T$.

	This conjecture is also empirically confirmed by the experiments. For example, in Figure \ref{fig:f1}, we plotted the complex eigenvalues, multiplied by $n^{-1/2}$, of $N$ realization of the polynomial $P_n(x)$ for different values of the triple $(k,n,N)$ under the constraint $knN=c$ for some positive integer $c$ (so that the number of the eigenvalues plotted is the same in every image). We display several subfigures organized as a matrix: the degree of the polynomial is constant on each row (namely $k=6$ for the first row and $k=4$ for the second row), while the columns are characterized by different values of $n$, increasing from left to right. To facilitate the visual comparison with the above claim, we also superimpose the unit circle on each image.

	\begin{figure}[h]
		\makebox[\textwidth][c]{\includegraphics[width=1.1\textwidth]{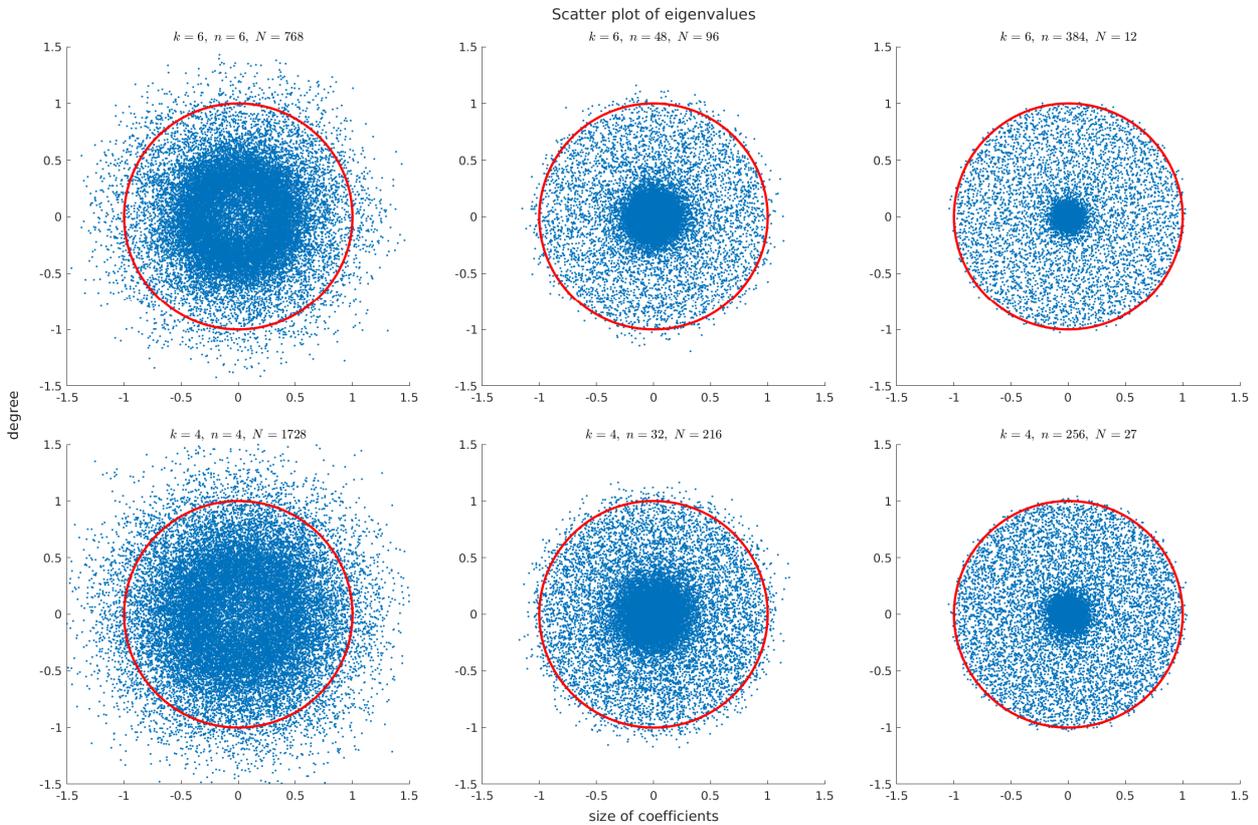}}%
		\caption{Scatter plots of the eigenvalues of $P_n(x)$ for growing $n$, multiplied by $\frac{1}{\sqrt{n}}$.}
		\label{fig:f1}
	\end{figure}
	
	We prove the claim as Theorem \ref{thm:esdmonicninf} below. Its proof relies on several technical lemmata on the behaviour of the singular values of certain matrices: in order to improve the readability of the paper, these are collected in Appendix \ref{sec:appbc} and Appendix \ref{sec:appn}. Since, for $k=1$, we recover the well known limit distribution of the eigenvalues of a Gaussian random matrix, within the proof we tacitly assume that $k \ge 2$.
	
	\begin{theorem}\label{thm:esdmonicninf}
		Let $P_n(x)$ be a monic $n \times n$ complex random matrix polynomial of degree $k$ as in \eqref{eq:Px}, where the entries of each coefficient $C_j$ are i.i.d. complex random variables normally distributed with mean $0$ and variance $1$. Then, for $n \rightarrow \infty$, the empirical spectral distribution of $P_n(n^{1/2}x)$ converges almost surely to
		\[ \frac{k-1}{k} {\bf 1}_0 +  \frac1k {\bf 1}_D,  \]
		where ${\bf 1}_0$, ${\bf 1}_D$ denote the uniform probability measures on, respectively, the set $\{ 0 \}$ and the unit disc.
	\end{theorem}
	
	\begin{proof}
		The strategy of the proof is to apply the Replacement Principle (Theorem \ref{thm:replacement}) in the special case where $m=kn$, $A_m =  M$ and $B_m =  E_1 C^T$, where $M,E_1,C$ are the matrices defined in \eqref{eq:C} and immediately below. Indeed, by the observations above, this immediately implies the statement. Thus, we need to verify that the two assumptions of Theorem \ref{thm:replacement} hold.
		\begin{enumerate}
			\item  Consider the random variable \[R_n = \frac{1}{k^2 n^2} \sum_{i=1}^{2kn^2} |X_i|^2,\] where $X_i$ are i.i.d. normally distributed complex random variables with mean $0$ and variance $1$. The $X_i$ depend also on $n$, and it is intended that all $X_i$ are i.i.d. for varying $i$ and $n$.
			Since 
			$\frac{1}{m^2} \left( \| A_m \|_F^2 + \| B_m \|_F^2 \right)$  has the same distribution as $R_n + \frac{k-1}{k^2n}$,
			it suffices to prove that $R_n$ is bounded almost surely. This is tantamount to $\Prob( \limsup_n R_n < \infty) =1$. On the other hand, by the strong law of large numbers, \[  \Prob\left( \lim_n \frac{1}{2 k n^2} \sum_{i=1}^{2kn^2} |X_i|^2 = \Expect [|X_i|^2] =1\right)=1;  \]
			it follows that  $\Prob\left( \limsup_n R_n < \infty \right) \ge \Prob \left(\limsup_n R_n = \frac2k\right)=1 .$
			\item Fix a nonzero complex number $w \neq 0$.  We need to verify that, for almost every $w$,
			\[   \frac{1}{kn} \left(  \log \left| \det   \left(  \frac{1}{\sqrt{kn}} E_1 C^T - wI  \right) \right|   - \log \left| \det  \left(  \frac{1}{\sqrt{kn}} M - wI  \right) \right|   \right) \xrightarrow{a.s.} 0. \]
			Defining $z:=w \sqrt{k}$ we readily see that this is equivalent to showing
			\begin{equation}\label{eq:eq1}
			\frac1n \sum_{i=1}^{kn} \left[ \log \sigma_i \left( \frac{1}{\sqrt{n}} E_1 C^T - z I  \right)  - \log \sigma_i \left( \frac{1}{\sqrt{n}}  M - z I  \right) \right] \xrightarrow{a.s.} 0
			\end{equation}   
			for every $z\ne 0$. Now let $0<\delta<1/2$ and set $f(n):=\lfloor kn-n^{1-\delta}
			\rfloor$. Observe that,
			for any $n$ large enough, $kn > f(n) > kn-n$. Rather than verifying \eqref{eq:eq1} directly, we will prove a somewhat stronger statement. Indeed, we claim that the following three facts all hold:
			\begin{equation}\label{eq:eq2}
			\frac{1}{n} \sum_{i=f(n)+1}^{kn}  \log \sigma_i \left( \frac{1}{\sqrt{n}}  M - z I  \right)  \xrightarrow{a.s.} 0. 
			\end{equation}
			\begin{equation}\label{eq:eq3}
			\frac{1}{n} \sum_{i=f(n)+1}^{kn}  \log \sigma_i \left( \frac{1}{\sqrt{n}} E_1 C^T - z I  \right)   \xrightarrow{a.s.} 0. 
			\end{equation}
			\begin{equation}\label{eq:eq4}
			\frac{1}{n} \sum_{i=1}^{f(n)} \left[ \log \sigma_i \left( \frac{1}{\sqrt{n}} E_1 C^T - z I  \right)  - \log \sigma_i \left( \frac{1}{\sqrt{n}} M - z I  \right) \right] \xrightarrow{a.s.} 0. 
			\end{equation}
			It is clear that \eqref{eq:eq2}, \eqref{eq:eq3} and \eqref{eq:eq4}, together, imply \eqref{eq:eq1}. It now remains to prove each statement separately.
			\begin{itemize}
				\item \emph{Proof of \eqref{eq:eq2}.} By Lemma \ref{lem:smallestsingM} and Lemma \ref{lem:spnAB}, 
				almost surely, for all $n$ sufficiently large, the following is true:
				\begin{align*}
				\sum_{i=f(n)+1}^{kn}\log
				\sigma_i\left( 
				\frac{1}{\sqrt {n}}
				M -zI
				\right)&\ge 
				\sum_{i=f(n)+1}^{kn}\log(
				n^{-a-2})\ge 
				(n^{1-\delta}+1)(-a-2) \log(n),\\
				\sum_{i=f(n)+1}^{kn}\log
				\sigma_i\left( 
				\frac{1}{\sqrt {n}}
				M -zI
				\right)
				&\le 
				\sum_{i=f(n)+1}^{kn}\log(d)
				\le (n^{1-\delta}+1)\log(d),
				\end{align*}
				where $a$ and $d$ are the positive constants appearing in Lemma \ref{lem:smallestsingM} and Lemma \ref{lem:spnAB}, and $d$ can be chosen greater than $1$. Hence, dividing by $n$,
				\begin{equation}\label{eq:thesisM_pr}
				(n^{-\delta} + n^{-1}) (-a-2) \log(n)
				\le \frac1n
				\sum_{i=f(n)+1}^{kn}\log
				\sigma_i\left( 
				\frac{1}{\sqrt {n}}
				M -zI
				\right)
				\le 
				(n^{-\delta} + n^{-1}) \log(d).
				\end{equation}
				Thus, \eqref{eq:eq2} follows by the sandwich rule.
				\item \emph{Proof of \eqref{eq:eq3}.} By Lemma \ref{lem:smallestsingE1CT} and Lemma \ref{lem:spnAB}, there are positive constants $\wt a$ and $d>1$ such that almost surely, for all $n$ sufficiently large,
				\begin{align*}
				\sum_{i=f(n)+1}^{kn}\log
				\sigma_i\left( 
				\frac{1}{\sqrt {n}}
				E_1C^T -zI
				\right) &\ge \sum_{i=f(n)+1}^{kn}\log(
				n^{-\wt a-2})\ge (n^{1-\delta}+1)(-\wt a-2) \log(n),\\
				\sum_{i=f(n)+1}^{kn}\log
				\sigma_i\left( 
				\frac{1}{\sqrt {n}}
				E_1C^T -zI
				\right) &
				\le 
				\sum_{i=f(n)+1}^{kn}\log(d)
				\le (n^{1-\delta}+1)\log(d).
				\end{align*}
				The latter inequalities imply
				\begin{equation}\label{eq:thesisE1CT_pr}
				(n^{-\delta} + n^{-1}) (-\wt a-2) \log(n)
				\le \frac1n
				\sum_{i=f(n)+1}^{kn}\log
				\sigma_i\left( 
				\frac{1}{\sqrt {n}}
				E_1C^T -zI
				\right)
				\le 
				(n^{-\delta} + n^{-1}) \log(d).
				\end{equation}
				yielding in turn \eqref{eq:eq3} via the sandwich rule.
				\item \emph{Proof of \eqref{eq:eq4}.}
				We start by the algebraic manipulation
				\[
				\frac 1{n} \sum_{i=1}^{f(n)} \left[\log
				\sigma_i\left( 
				\frac{1}{\sqrt {n}}
				E_1 C^T -zI
				\right)
				-
				\log \sigma_i\left( 
				\frac{1}{\sqrt {n}}
				M -zI
				\right)
				\right]
				=
				-\frac 1{n} \sum_{i=1}^{f(n)} \left[
				\log
				\frac{\sigma_i\left( 
					\frac{1}{\sqrt {n}}
					M -zI
					\right)}{\sigma_i\left( 
					\frac{1}{\sqrt {n}}
					E_1 C^T -zI
					\right)}
				\right].
				\]
				Thanks to Mirsky's Theorem (\autoref{thm:Perturbation}), we know that, for every $i$,
				\[
				\left|\sigma_i\left( 
				\frac{1}{\sqrt {n}}
				M -zI
				\right)
				-
				\sigma_i\left( 
				\frac{1}{\sqrt {n}}
				E_1 C^T -zI
				\right)\right|
				\le \frac{1}{\sqrt {n}}\| M - E_1 C^T\| =  \frac{1}{\sqrt {n}}
				\]
				so, for $i=1,\dots,f(n)$ there exist $d_i$ satisfying $|d_i| \le \frac{1}{\sqrt {n}}$ and such that
				\[
				\sigma_i\left( 
				\frac{1}{\sqrt {n}}
				M -zI
				\right)
				=
				\sigma_i\left( 
				\frac{1}{\sqrt {n}}
				E_1 C^T -zI
				\right)
				+d_i.
				\]
				Thus,
				\begin{equation}\label{eq:manip}
				\left|\frac 1{n} \sum_{i=1}^{f(n)} \left[\log
				\sigma_i\left( 
				\frac{1}{\sqrt {n}} 
				E_1 C^T -zI
				\right)
				-
				\log \sigma_i\left( 
				\frac{1}{\sqrt {n}}
				M -zI
				\right)
				\right]\right|
				\le 
				\frac 1{n} \sum_{i=1}^{f(n)}
				\left|
				\log\left(
				1 + 
				\frac{d_i}{\sigma_i\left( 
					\frac{1}{\sqrt {n}}
					E_1 C^T -zI
					\right)}
				\right)
				\right|
				\end{equation}
				Observe now that, using Lemma \ref{lem:sigmafnE1CT}, we have that, for some positive constants $t,\ve$, almost surely, for all $n$ sufficiently large and for every $i\le f(n)$, 
				\[ |x| := \left|\frac{d_i}{\sigma_i(n^{-1/2}E_1 C^T-zI)}\right|
				\le 
				\left|\frac{n^{-1/2}}{\sigma_{f(n)}(n^{-1/2}E_1 C^T-zI)}\right|
				\leq t^{-1}n^{-\ve} .  \]
				For sufficiently large $n$ (i.e. $n > t^{-1/\ve}$), the right hand side of the latter inequality is bounded above by $1$. Noting that $|x| < 1 \Leftrightarrow |\log(1+x)| \leq - \log(1-|x|)$, we obtain the following upper bound for the right hand side of \eqref{eq:manip}:
				\[
				0\le 
				-
				\frac 1{n} \sum_{i=1}^{f(n)}
				\log \left(
				1 -
				\frac{|d_i|}{\sigma_i\left( 
					\frac{1}{\sqrt {n}}
					E_1 C^T -zI
					\right)}
				\right)
				\le 
				-
				\frac {f(n)}{n} 
				\log\left(
				1 -
				\frac{n^{-1/2}}{\sigma_{f(n)}\left( 
					\frac{1}{\sqrt {n}}
					E_1 C^T -zI
					\right)}
				\right) \]
				which in turn is bounded above by 
				\[-
				k
				\log\left(
				1 -
				\frac{n^{-1/2}}{\sigma_{f(n)}\left( 
					\frac{1}{\sqrt {n}}
					E_1 C^T -zI
					\right)}
				\right).
				\]
				Invoking again Lemma \ref{lem:sigmafnE1CT}, we have that almost surely
				\[
				0\le
				\frac{n^{-1/2}}{\sigma_{f(n)}\left( 
					\frac{1}{\sqrt {n}}
					E_1 C^T -zI
					\right)}
				\le t^{-1}n^{-\ve}\rightarrow 0
				\]
				\[
				\implies
				-
				k
				\log\left(
				1 -
				\frac{n^{-1/2}}{\sigma_{f(n)}\left( 
					\frac{1}{\sqrt {n}}
					E_1 C^T -zI
					\right)}
				\right) 
				\xrightarrow{a.s.} 0,
				\]
				and this concludes the proof.
			\end{itemize}
		\end{enumerate}
	\end{proof}
	
	\begin{remark}
		The relations \eqref{eq:eq2} and \eqref{eq:eq3} still hold if the entries of $C_i$ are i.i.d.\ copies of any centered random variable with unit variance, using slight variations of the reported results. 
	\end{remark}

	\section{Empirical spectral distribution for  $n \times n$ monic complex Gaussian matrix polynomials of degree $k$ in the limit $k \rightarrow \infty$}\label{sec:k}
	Consider again\footnote{The slight notational change with respect to \eqref{eq:Px} is just to emphasize that here we will let $k \rightarrow \infty$ rather than $n \rightarrow \infty$.} the monic matrix polynomial
	\begin{equation}\label{eq:Px2}
	P_k(x) = I_n x^k + \sum_{j=0}^{k-1} C_j x^j,
	\end{equation}
	so that for all $j=0,\dots,k-1$ every coefficient $C_j$ is a $n\times n$ random matrix where all the entries are i.i.d.\ Gaussian complex random variables with mean zero and variance 1. 
	Note that each $C_j$ depends on $j$ and on $k$, but we omit the dependence on $k$ in the notation. It is intended moreover that all $C_j$ are independent of each other for varying $j$ and $k$. 
	
	The finite eigenvalues of $P_k(x)$ coincide with those of its companion matrix $M$ as in \eqref{eq:C}. However, this time we decompose $M$ as the sum of a deterministic circulant matrix and a random matrix with rank at most $n$
	\begin{equation}\label{eq:matrix_split_k}
	M= B + A,\qquad 
	B = 
	\begin{bmatrix}
	&  &  & I_n\\
	I_n & & & \\
	& \ddots & & \\
	& & I_n & 
	\end{bmatrix},\qquad A = \begin{bmatrix}
	-C_{k-1} & \dots & -C_1 & -(C_0+I_n)\\
	\phantom{I_n} & & & \\
	& \phantom{\ddots}  & & \\
	& & \phantom{I_n}  & 
	\end{bmatrix} = E_1 \wh C^T,
	\end{equation}
	where $\wh C^T = C^T - e_k^T\otimes I_n$,  $E_1^T = \begin{bmatrix}
	I_n & 0 & \dots & 0
	\end{bmatrix}$ and $C^T = -\begin{bmatrix}
	C_{k-1} & \dots & C_1 & C_0
	\end{bmatrix}$.
	In particular, $B$ is a circulant matrix \cite{DavisBook}, with spectrum
	\[
	\Lambda(B) = \set{ \exp( 2\pi\textnormal ij/k  ) | j=0,1,\dots,k-1  }
	\]
	where each eigenvalue has multiplicity $n$. It is thus easy to see that the almost sure limit ESD of $B$ is the uniform (singular) probability measure on the unit circle $\bf 1_U$. 
	The problem can be seen through the lens of the theory of perturbations for Toeplitz matrices and sequences (see for example \cite{BPZ19,VZ21}), but since the perturbation $E_1\hat C^T$ is a rank $n$ correction to the Teoplitz matrix $B$, this case does not fall in the classical settings, where it is required that the random perturbation is not singular with high probability. 
	Since the rank of the perturbation is small when compared with the growing size $kn$ of the matrices, we may anyway expect that the ESD of $M$ also converge almost surely to the same distribution $\bf 1_U$. 
	
	This claim is also empirically supported by the experiments. For example, in Figure \ref{fig:f2} we plot the complex eigenvalues of $N$ realization of the polynomial $P_k(x)$ for different values of the triple $(k,n,N)$. The interpretation of Figure \ref{fig:f2} is the same as Figure \ref{fig:f1} after swapping the roles of the matrix sizes and the degrees: we fix $n$ on each row (to $6$ and $4$ respectively) and we increase the degree of the polynomial on the columns. The unit circle is drawn on top of each scatter plot to make easier the comparison with the claim above; $N$ is  always chosen so that the number of eigenvalues plotted, equal to $knN$, is the same in every image.
	
	\begin{figure}[h]
		\makebox[\textwidth][c]{\includegraphics[width=1.1\textwidth]{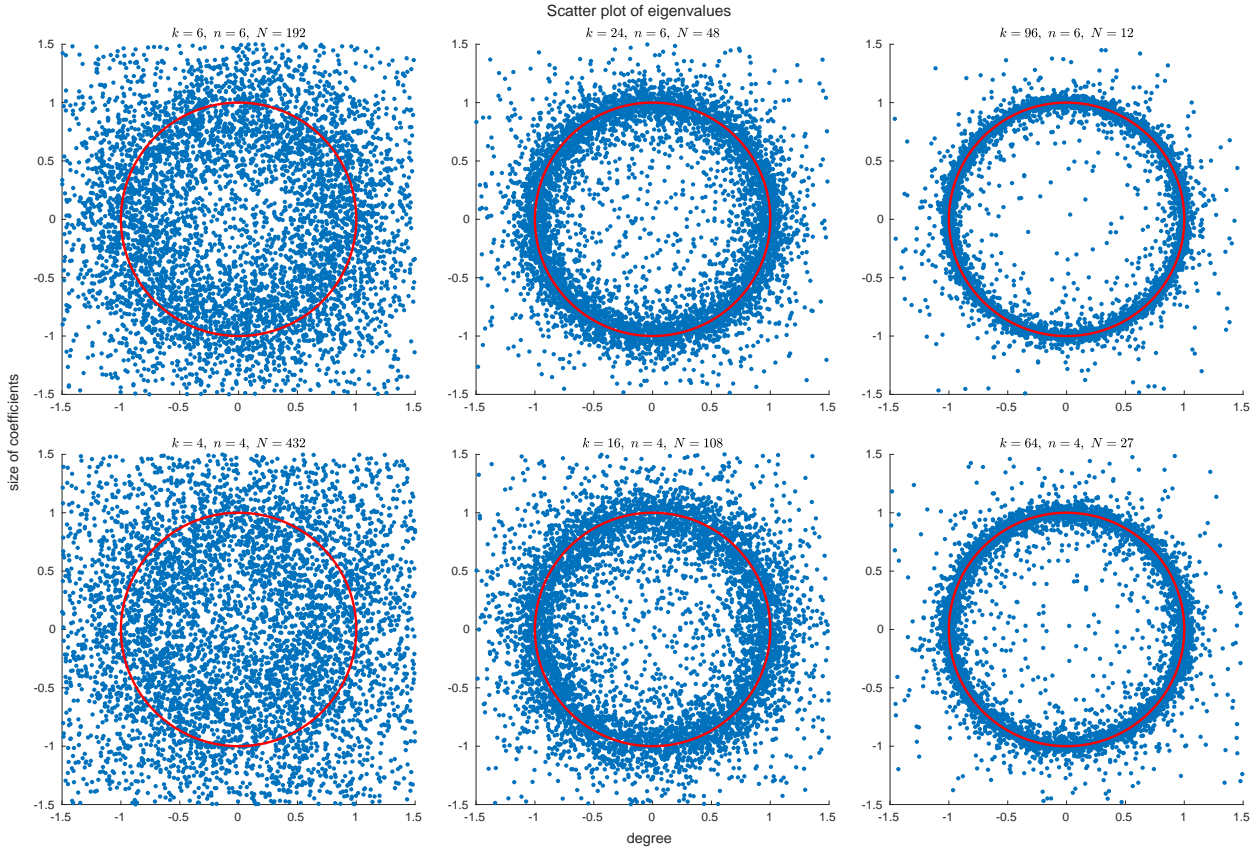}}%
		\caption{Scatter plots of the eigenvalues of $P_k(x)$ for growing $k$.}
		\label{fig:f2}
	\end{figure}
	
	In Theorem \ref{thm:esdmonicninf} below, we thus show that the ESD of  $M$ also converge almost surely to $\bf 1_U$. Note that the statement includes, as a special case when $n=1$, the well known limit distribution of random scalar polynomials, for which we thus provide a novel proof. For the sake of a clearer exposition, below we focus on the major lines of thought that lead to the proof, postponing to Appendix \ref{sec:appk} a more detailed analysis of some technicalities that appear as intermediate steps.
	
	\begin{theorem}\label{thm:esdmonicnk}
		Let $P_k(x)$ be a monic $n \times n$ complex random matrix polynomial of degree $k$ as in \eqref{eq:Px}, where the entries of each coefficient $C_j$ are i.i.d. complex random variables normally distributed with mean $0$ and variance $1$. Then, for $k \rightarrow \infty$, the empirical spectral distribution of $P_k(x)$ converges almost surely to $\bf 1_U$,
		the uniform probability measure on the unit circumference.
	\end{theorem}
	
	\begin{proof}
		The strategy of the proof follows very closely that of Theorem \ref{thm:esdmonicninf}: we verify that the two assumptions of Theorem \ref{thm:replacement} hold in the special case where $m=kn$, $A_m = \sqrt{kn} \ M$ and $B_m =\sqrt{kn} \ B $, where $M,B$ are the matrices defined in \eqref{eq:matrix_split_k} and immediately below.
		\begin{enumerate}
			\item 
			The first item is treated analogously to Theorem \ref{thm:esdmonicninf}, and we omit the details.
			\item Fix a nonzero complex number $z$ such that $|z|\not\in \set{0,1}$.  We  show that, for  every such $z$,
			\[   \frac{1}{kn} \left(  \log \left| \det   \left(  M - zI  \right) \right|   - \log \left| \det  \left(  B - zI  \right) \right|   \right) \xrightarrow{a.s.} 0, \]
			that is equivalent to
			\begin{equation}\label{eq:2eq1}
			\frac 1{k} \sum_{i=1}^{kn} \log(
			\sigma_i\left( 
			M -zI
			\right))
			-
			\log(\sigma_i\left( 
			B -zI
			\right))
			\xrightarrow{a.s.} 0.
			\end{equation}
			We claim that the following facts are true:
			\begin{equation}\label{eq:thesis_k1}
			\frac 1{k} \sum_{i=1}^{n}\log(
			\sigma_i\left( 
			M -zI
			\right))
			\xrightarrow{a.s.} 0,
			\qquad 
			\frac 1{k} \sum_{i=kn -n+1}^{kn}\log(
			\sigma_i\left( 
			M -zI
			\right))
			\xrightarrow{a.s.} 0.
			\end{equation}
			\begin{equation}\label{eq:thesis_k2}
			\frac 1{k} \sum_{i=1}^{n}\log(
			\sigma_i\left( 
			B -zI
			\right))
			\xrightarrow{a.s.} 0,
			\qquad 
			\frac 1{k} \sum_{i=kn -n+1}^{kn}\log(
			\sigma_i\left( 
			B -zI
			\right))
			\xrightarrow{a.s.} 0.
			\end{equation}
			\begin{equation}\label{eq:thesis_k3}
			\frac 1{k} \sum_{i=n+1}^{kn -n} \log(
			\sigma_i\left( 
			M -zI
			\right))
			-
			\log(\sigma_i\left( 
			B -zI
			\right))
			\xrightarrow{a.s.} 0.
			\end{equation}
			It is clear that \eqref{eq:thesis_k1}, \eqref{eq:thesis_k2} and \eqref{eq:thesis_k3},
			together, imply  \eqref{eq:2eq1}.  It now remains to prove each statement separately.
			\begin{itemize}
				\item \emph{Proof of \eqref{eq:thesis_k1}.}
				By Lemma \ref{lem:greatestsvA}, 
				almost surely, for all $k$ sufficiently large, the following are true almost surely for some positive constant $r$:
				\[
				\sigma_{1}(M-zI)  \le   r\sqrt k + 1 + |z|,
				\qquad 
				\sigma_n(M-zI)\ge|1-|z||.
				\]
				These facts are enough to conclude that
				\[
				\frac 1{k}\sum_{i=1}^{n}|\log(
				\sigma_i\left( 
				M -zI
				\right))|
				\le 
				\frac n{k}
				\max\left\{
				|\log(r\sqrt k + 1 + |z|)|,
				|\log( |1-|z||  )|
				\right\}
				\xrightarrow{a.s.} 0
				\]
				and thus the first a.s. limit in \eqref{eq:thesis_k1} holds. Moreover, by Lemma
				\ref{lem:smallestsvA},
				almost surely, for all $k$ sufficiently large,
				\[ \sigma_{kn} \left( M - z I \right) \ge tk^{-2}  \]
				for some positive constant $t$. Hence, it suffices to estimate
				\[
				\frac 1{k} \sum_{i=kn -n+1}^{kn}|\log(
				\sigma_i\left( 
				M -zI
				\right)) |
				\le 
				\frac{n}{k}\max\{
				|\log(\sigma_1(M -zI))|,
				|\log(\sigma_{kn}(M -zI))|
				\}
				\]
				\[
				\le 
				\frac nk\max\{
				|\log( r\sqrt k + 1 +  |z|)|,
				|\log(tk^{-2})|
				\}\xrightarrow{a.s.} 0.
				\]
				Thus, the second part of \eqref{eq:thesis_k1} also holds.

				\item \emph{Proof of \eqref{eq:thesis_k2}.}
				Observe that $B-zI$ is a circulant matrix, and hence, in particular it is normal. Its spectrum is
				\[
				\Lambda(B-zI) = \set{ \lambda - z | \lambda\in \Lambda(B)}.
				\]
				Since all eigenvalues of $B$ have unitary norm, we can bound the singular values of $B-zI$ as
				\begin{equation}\label{eq:bound_sv_B}
				\sigma_i(B-zI) = |\lambda_i - z|,
				\qquad
				| 1 - |z||  \le |\lambda_i - z| \le 1 + |z|.
				\end{equation}
				Importantly, these bounds do not depend on $k$. As a consequence,
				\[\frac
				{n\log(| 1 - |z|| )}{k}
				\le 
				\frac 1{k} \sum_{i=kn -n+1}^{kn}\log(
				\sigma_i\left( 
				B -zI
				\right)) \le \frac {n\log(1 + |z|)}{k},
				\]
				\[\frac
				{n\log(| 1 - |z|| )}{k} 
				\le 
				\frac 1{k} \sum_{i=1}^{n}\log(
				\sigma_i\left( 
				M -zI
				\right)) \le  \frac {n\log(1 + |z|)}{k},
				\]
				and \eqref{eq:thesis_k2} follows by the sandwich rule.
				\item \emph{Proof of \eqref{eq:thesis_k3}.}
				We start by noting that the statement is implied by
				\[
				\frac 1{k} \sum_{i=n+1}^{kn -n} \left|\log
				\left(
				\frac{\sigma_i\left( 
					M -zI
					\right)}{\sigma_i\left( 
					B -zI
					\right)}
				\right)
				\right|\xrightarrow{a.s.} 0.
				\]
				Assume now that $k>2$. Observe that $M-zI$ is a perturbation of rank at most $n$ of $B-zI$. As a consequence, by \autoref{thm:Interlacing2},  we find that
				\begin{equation}\label{eq:interlacing_AB}
				\sigma_{i+n}(B-zI)\le
				\sigma_i(M-zI) \le \sigma_{i-n}(B-zI)
				\end{equation}
				for every $n<i\le nk-n$. Thus, 
				\begin{equation*}
				\frac 1{k} \sum_{i=n+1}^{kn -n} \left|\log
				\left(
				\frac{\sigma_i\left( 
					M -zI
					\right)}{\sigma_i\left( 
					B -zI
					\right)}
				\right)
				\right|\le
				\frac 1{k} \sum_{i=n+1}^{kn -n}
				\max \left\{
				\left|\log
				\left(
				\frac{\sigma_{i-n}\left( 
					B -zI
					\right)}{\sigma_i\left( 
					B -zI
					\right)}
				\right)
				\right|,
				\left|\log
				\left(
				\frac{\sigma_{i+n}\left( 
					B -zI
					\right)}{\sigma_i\left( 
					B -zI
					\right)}
				\right)
				\right|
				\right\}
				\end{equation*}
				The singular values of $B-zI$ are the moduli of $\lambda_i-z$ where $\lambda_i$ are the eigenvalues of $B$. For the rest of this argument, and for the sake of notational simplicity, let us now drop the dependence on the argument matrix and simply refer to the $r$th singular value of $B-zI$ as $\sigma_r$.  Since all the eigenvalues of $B$ have multiplicity $n$, then
				$ \sigma_{i-n}= |\lambda_j -z|$,  
				$ \sigma_{i} = |\lambda_i -z|$ and $ \sigma_{i+n} = |\lambda_s -z|$, where necessarily $i,j,s$ are pairwise distinct; specifically, $j$ and $s$ are determined by $z$ coherently with the decreasing ordering of the singular values.  We conclude that $\sigma_{i} - \sigma_{i+n}$ and $\sigma_{i-n} - \sigma_{i}$ are both bounded above by
				\[
				\min_{j \neq s \neq i \neq j} \max \{ ||\lambda_i -z| - |\lambda_j -z||, ||\lambda_i -z| - |\lambda_s -z||     \}
				\le 
				\min_{j \neq s \neq i \neq j} \max \{ |\lambda_i -\lambda_j|, |\lambda_i -\lambda_s | .    \}
				\]
				In particular, as $k>2$, we can choose 
				\[
				\lambda_j = \lambda_i \exp(2\pi\textnormal i/k),\qquad \lambda_s = \lambda_i \exp(-2\pi\textnormal i/k),
				\]
				and hence,
				\[
				\min_{j \neq s \neq i \neq j} \max \{ |\lambda_i -\lambda_j|, |\lambda_i -\lambda_s |     \}
				\le 
				|1-\exp(2\pi\textnormal i/k)| = 2\sin(\pi/k).
				\]
				Therefore, for all values of $k$ large enough so that $0<2\sin(\pi/k)<|1-|z||$, we use \eqref{eq:bound_sv_B} to obtain
				\[
				\left|\log
				\left(
				\frac{\sigma_{i+n}}{\sigma_i}
				\right)
				\right|
				=
				- \log\left( 1 - \frac{\sigma_i -\sigma_{i+n}}{\sigma_i}\right)
				\le 
				- \log\left( 1 - 2\frac{\sin(\pi/k)}{|1-|z||}\right),
				\]
				\[
				\left|\log
				\left(
				\frac{\sigma_{i-n}}{\sigma_i}
				\right)
				\right|
				=
				\log
				\left(
				1+
				\frac{\sigma_{i-n} -\sigma_{i}}{\sigma_i}
				\right)
				\le 
				\log\left( 1 + 2\frac{\sin(\pi/k)}{|1-|z||}\right).
				\]
				and, since $0<x<1 \implies -\log(1-x) > \log(1+x)$, we conclude that
				\[
				\frac 1{k} \sum_{i=n+1}^{kn -n}
				\max \left\{
				\left|\log
				\left(
				\frac{\sigma_{i-n}\left( 
					B -zI
					\right)}{\sigma_i\left( 
					B -zI
					\right)}
				\right)
				\right|,
				\left|\log
				\left(
				\frac{\sigma_{i+n}\left( 
					B -zI
					\right)}{\sigma_i\left( 
					B -zI
					\right)}
				\right)
				\right|
				\right\}
				\le 
				-\frac {kn-2n}{k} 
				\log\left( 1 - 2\frac{\sin(\pi/k)}{|1-|z||}\right)
				\]
				that goes to zero as $k\to\infty$, implying \eqref{eq:thesis_k3}.
			\end{itemize}
		\end{enumerate}

	\end{proof}

	\section{Conclusions}\label{sec:concl}
	
	We have rigorously obtained the limit of empirical spectral distribution for monic complex i.i.d. Gaussian matrix polynomials. To our knowledge, and in spite of the relatively common use of random matrix polynomials in the context of numerical experiments to test algorithms for the polynomial eigenvalue problem, the study in the present paper is the first attempt to study analytically the distribution of eigenvalues of a class of random matrix polynomials.
	
	We hope that this work may open the path to further future research on eigenvalues of random matrix polynomials. In particular we believe that it would be of interest to extend our results by considering, for instance, different ways to send $k,n \rightarrow \infty$, non-monic polynomials, coefficients restricted to be real (and/or otherwise structured) and more general distributions of the entries.
	
	In a forthcoming document, the authors will show further progress about proving the convergence of non-monic non-Gaussian polynomial empirical spectral distribution in both cases $n,k\to\infty$.
	
	\subsection*{Acknowledgement}
	
	We acknowledge the computational resources provided by the Aalto Science-IT project.
	This version of the article has been accepted for publication, after peer review, but is not the Version of Record and does not reflect post-acceptance improvements, or any	corrections. The Version of Record is available online at: \href{https://doi.org/10.1007/s10959-022-01163-3 }{https://doi.org/10.1007/s10959-022-01163-3}.
	
	%
	%
	
	\bibliographystyle{abbrv}
	\bibliography{randomthings}

\begin{thebibliography}{10}

\bibitem{AT12}
M.~Al-Ammari and F.~Tisseur.
\newblock Standard triples of structured matrix polynomials.
\newblock {\em Linear Algebra Appl.}, 437(3):817--834, 2012.

\bibitem{AB19}
D.~Armentano and C.~Beltr{\'a}n.
\newblock The polynomial eigenvalue problem is well conditioned for random
  inputs.
\newblock {\em SIAM J. Matrix Anal. Appl.}, 40(1):175--193, 2019.

\bibitem{BPZ19}
A.~Basak, E.~Paquette, and O.~Zeitouni.
\newblock Spectrum of random perturbations of toeplitz matrices with finite
  symbols.
\newblock {\em Transactions of the American Mathematical Society}, 373:1, 11
  2019.

\bibitem{BK19}
C.~Beltr{\'a}n and K.~Kozhasov.
\newblock The real polynomial eigenvalue problem is well conditioned on the
  average.
\newblock {\em Found. Comput. Math.}, 20(2):291--309, 2020.

\bibitem{BC10}
P.~B\"{u}rgisser and F.~Cucker.
\newblock Smoothed analysis of {M}oore-{P}enrose inversion.
\newblock {\em SIAM J. Matrix Anal. Appl.}, 31(5):2769--2783, 2010.

\bibitem{BCBook}
P.~B\"{u}rgisser and F.~Cucker.
\newblock {\em Condition. The Geometry of Numerical Algorithms}.
\newblock Grundlehren der mathematischen Wissenschaften. Springer-Verlag Berlin
  Heidelberg, 2013.

\bibitem{DavisBook}
P.~J. Davis.
\newblock {\em Circulant Matrices}.
\newblock AMS Chelsea Publishing. Cambridge University Press, 1994.

\bibitem{DLPVD18}
F.~Dopico, P.~W. Lawrence, J.~P\'{e}rez, and P.~{Van Dooren}.
\newblock Block {K}ronecker linearizations of matrix polynomials and their
  backward errors.
\newblock {\em Numer. Math.}, 140:373--426, 2018.

\bibitem{Dopico2018}
F.~Dopico and V.~Noferini.
\newblock Root polynomials and their role in the theory of matrix polynomials.
\newblock {\em Linear Algebra Appl.}, 584:37--78, 2020.

\bibitem{GLRBook}
I.~Gohberg, P.~Lancaster, and L.~Rodman.
\newblock {\em Matrix Polynomials}.
\newblock SIAM, 2009.
\newblock Unabridged republication of the book first published by Academic
  Press.

\bibitem{GT17}
S.~G\"{u}ttel and F.~Tisseur.
\newblock The nonlinear eigenvalue problem.
\newblock {\em Acta Numer.}, 26:1--94, 2017.

\bibitem{LN20}
M.~Lotz and V.~Noferini.
\newblock Wilkinson's bus: Weak condition numbers, with an application to
  singular polynomial eigenproblems.
\newblock {\em Found. Comput. Math.}, 2020.

\bibitem{MehtaBook}
M.~L. Mehta.
\newblock {\em Random matrices and the statistical theory of energy levels}.
\newblock Academic Press, New York, 1967.

\bibitem{MIR60}
L.~Mirsky.
\newblock {Symmetry Gauge Functionsa and Unitarily Invariant Norms}.
\newblock {\em The Quarterly Journal of Mathematics}, 11(1):50--59, 01 1960.

\bibitem{NP15}
V.~Noferini and F.~Poloni.
\newblock Duality of matrix pencils, {W}ong chains and linearizations.
\newblock {\em Linear Algebra Appl.}, 471:730--767, 2015.

\bibitem{TaoBook}
T.~Tao.
\newblock {\em Topics in Random Matrix Theory}.
\newblock Graduate studies in mathematics. American Mathematical Soc.

\bibitem{TV08}
T.~Tao and V.~Vu.
\newblock Random matrices: The circular law.
\newblock {\em Commun. Contemp. Math.}, 10(02):261--307, 2008.

\bibitem{TVK10}
T.~Tao, V.~Vu, and M.~Krishnapur.
\newblock Random matrices: Universality of {E}{S}{D}s and the circular law.
\newblock {\em Ann. Probab.}, 38(5):2023--2065, 2010.

\bibitem{THO72}
R.~Thompson.
\newblock Principal submatrices {I}{X}: Interlacing inequalities for singular
  values of submatrices.
\newblock {\em Linear Algebra and its Applications}, 5(1):1 -- 12, 1972.

\bibitem{THO76}
R.~Thompson.
\newblock The behavior of eigenvalues and singular values under perturbations
  of restricted rank.
\newblock {\em Linear Algebra and its Applications}, 13(1):69 -- 78, 1976.

\bibitem{TM01}
F.~Tisseur and K.~Meerbergen.
\newblock The quadratic eigenvalue problem.
\newblock {\em SIAM Rev.}, 43(2):235--286, 2001.

\bibitem{VZ21}
M.~Vogel and O.~Zeitouni.
\newblock Deterministic equivalence for noisy perturbations.
\newblock {\em Proceedings of the American Mathematical Society},
  ff10.1090/proc/15499ff, 2021.

\bibitem{WOOD50}
M.~A. Woodbury.
\newblock {Inverting Modified Matrices}.
\newblock Technical Report Memorandum Report 42, Department of Statistics,
  Institute for Advanced Study, Princeton University, 1950.

\end{thebibliography}
	
	\appendix
	\section{Appendix: Technical results}\label{sec:app}
	In this appendix, we provide the full details on some technical steps that are necessary for our analysis. For convenience, we have split the appendix into various subsections, according to the specific nature of the results contained therein.

	\subsection{Preliminaries and known results}\label{sec:apppre}

	\noindent
	A result we frequently use in our arguments is the following interlacing property of the singular values of submatrices.
	\begin{theorem}[Interlacing Singular Values for Submatrices \cite{THO72}]\label{cor:Interlacing}
		Given any matrix $A$ and any $p\times q$ submatrix $B$,\[
		\|A\| \ge \|B\|, \qquad \sigma_{\min(p,q)}(A) \ge \sigma_{\min} (B).
		\]
	\end{theorem}

	Moreover, we recall the useful Woodbury identity.
	\begin{lemma}[Woodbury \cite{WOOD50}]\label{lem:woodbury}
		Let $A,B,U,V$ be complex matrices satisfying $B = A + UV$ with $A,B$ square. If $A$ and $I + VA^{-1}U$ are both invertible, then
		\[
		B^{-1} = A^{-1} - A^{-1}U(I + VA^{-1}U)^{-1}VA^{-1}.
		\]
	\end{lemma}

	When dealing with sequences of random matrices, one can estimate the distribution of eigenvalues, singular values, or related quantities such as, for instance, trace, determinant, norms. Here we collect some of the estimations we use further on. We do not claim that the bounds we mention below are the best possible ones; yet, they suffice for our purposes. \\
	
	\noindent First, we provide a probabilistic upper bound for the norm of a Gaussian random matrix.
	
	\begin{theorem}[\cite{TaoBook}]\label{thm:norm_bound}
		Suppose that the coefficients of a random matrix $N$ of size $\sqsize n$ are i.i.d.\ copies of a normal random variable. Then there exist absolute
		constants $C,c>0$ such that
		\[
		\f P\left( \| N\| >A\sqrt n\right) \le C \exp(-cAn)
		\]
		for all $A\ge C$. 
	\end{theorem}
	\noindent
	A very different kind of estimate, due to Tao and Vu, is needed for the least singular value of random matrices having nonzero mean.

	\begin{theorem}[\cite{TV08}]\label{thm:least_sv_bound}
		Let $c$, $d$ be positive constants, and let
		$X$ be a complex-valued random variable with non-zero finite variance. Then there are positive constants $a$ and $b$ such that
		the following holds: if $N_n$ is the $n \times n$ random matrix whose entries are i.i.d.
		copies of $X$, and $M$ is an $n \times n$ deterministic matrix with spectral norm at most
		$n^{c}$, then,
		\[
		\f P\left( \|(M+N_n)^{-1}\|\ge n^a  \right)
		\le 
		bn^{-d}.
		\]
	\end{theorem}
	\noindent Note that, in Theorem \ref{thm:least_sv_bound}, we can always choose $d>1$ so that the probability is summable. Thus, we can use Borel-Cantelli lemma to obtain a lower bound for the last singular value valid for all sufficiently large $n$.\\

	A widely used distribution in the theory of probability for real random variables is the Beta distribution on $[0,1]$, that depends on two positive parameters $\alpha, \beta$ and it is described by its density function
	\[
	Y\sim B(\alpha,\beta) \implies 
	\f P \left( Y\le \lambda \right) = \frac{\Gamma(\alpha + \beta)}{\Gamma(\alpha)\Gamma(\beta)} \int_0^\lambda
	x^{\alpha-1}(1-x)^{\beta-1}\, {{\rm d}}x,\quad \forall \, 0\le \lambda\le 1,
	\]
	where $\Gamma(z) = \int_0^\infty t^{z-1} e^{-t} dt$ is Euler's gamma function. It is known that if $z$ is a positive integer, then $\Gamma(z) = (z-1)!$. 
	
	Consider now a real random vector $X$ with $N$ components, uniformly distributed on the unit sphere $S^{N-1}:=\{ X \in \R^N : \| X \|_2 =1 \}$. It is known \cite[Sec. 4.1.1]{LN20} that the squared norm of the projection of $X$ onto a $k$-dimensional space is Beta distributed with parameters $B(k/2,(N-k)/2)$. Note that a complex random vector with $N$ complex components, uniformly distributed on the respective complex spherical surface, can be seen as a real vector with $2N$ real components and uniformly distributed on $S^{2N-1}$. As a consequence,  the squared norm of the projection of such a vector onto a $k$-dimensional complex space is Beta distributed with parameters $B(k,N-k)$.
	In particular, if $k=1$, then 
	\begin{equation}\label{eq:projection_beta}
	\f P\left( \|P_1(X)\|^2\le \lambda \right) = \frac{\Gamma(N)}{\Gamma(1)\Gamma(N-1)} \int_0^\lambda
	(1-x)^{N-2}\, {{\rm d}}x
	=
	1 - (1-\lambda)^{N-1}
	,\quad \forall \ \lambda \in [0,1].
	\end{equation}

	\subsection{A variation on a result by B\"{u}rgisser and Cucker: a formula for the tail bounds of the norm of the pseudoinverse of a non-zero mean random matrix}\label{sec:appbc}
	
	Theorem \ref{thm:normpseudoinverse} yields a tail bound on the norm of the Moore-Penrose pseudoinverse of a random Gaussian complex rectangular matrix \emph{with nonzero mean}. It is a modification of the results obtained by B\"{u}rgisser and Cucker in \cite[Sec. 3]{BC10} and \cite[Ch. 4]{BCBook}, with two differences. A minor one is that we work with complex, as opposed to real, numbers and random variables (as noted already in \cite{BCBook}, this extension is not at all difficult). A more significant one is that we are interested in the limit case where $\lambda = \frac{n-1}{N} \rightarrow 1$, and we therefore state the result in such a way that it covers that case, unlike \cite{BC10,BCBook} that provide formulae for the regime $\lambda<1$. For these reasons, as well as for the sake of self-containedness, we provide a full proof (which still follows very closely the lead of \cite{BC10,BCBook}).
	
	\begin{theorem}\label{thm:normpseudoinverse}
		Let $G$ be an $n \times N$ (with $N \ge n$) complex random matrix with i.i.d. normally distributed entries with mean $0$ and variance $n^{-1}$. Suppose $R=R_D + G$ where $R_D \in \C^{n \times N}$ is a deterministic matrix, and let $R^\dagger$ be the Moore-Penrose pseudoinverse of $R$. Then, given $\tau > 0$,
		\[
		\Prob \left( \sigma_n(R)\le \tau  \right)
		\le 
		\frac { \tau^{2(N-n+1)}}{\sqrt{2\pi} }
		\frac{
			(nNe^2)^{N-n+1}
		}
		{
			(N-n+1)^{2(N-n+1) + 1/2}
		}.
		\]
	\end{theorem}
	
	\begin{proof}
		We know that there exists an unit vector $u \in \C^n$ such that
		\[
		\|R^\dagger u\| = \|R^\dagger\|
		\]
		and that, for almost every\footnote{The set of matrices $R$ can be parametrized by a random vector in $\R^{2nN}$. The subset of matrices for which this property fails corresponds to a subset of the proper algebraic set described by $\mathrm{discriminant}(\det(RR^* - x I))=0$.
		}  $R$, $u$ is unique up to multiplication by a unit of $\C$. If $v \in \C^n$ is any unit vector, then for some $\beta \in \C$ it can be decomposed as 
		\[
		v = (u^*v)u +\beta u^\perp
		\]
		where $u^\perp \in \C^n$ is a unit vector orthogonal to $u$. On the other hand, since $R^\dagger u$ is orthogonal to $R^\dagger u^\perp$, we have
		\[
		\|R^\dagger v\|^2 = | u^*v|^2\|R^\dagger u\|^2 + \|\beta R^\dagger u^\perp\|^2
		\ge  |u^*v|^2\|R^\dagger\|^2
		\]
		\[
		\implies \|R^\dagger v\|\ge |u^*v| \|R^\dagger\|.
		\]
		Thus, for any $s\in(0,1)$ and $t>0$ we have
		\[
		\Prob_{R,v} \left( \|R^\dagger v\|\ge t\sqrt{1-s^2}  \right)
		\ge 
		\Prob_{R,v} \left( |u^*v|\ge \sqrt{1-s^2} \text{ and } \|R^\dagger\|\ge t \right)
		\]
		or equivalently
		\[
		\Prob_{R,v} \left( \|R^\dagger v\|\ge t\sqrt{1-s^2}  \right) \ge
		\Prob_{R} \left(  \|R^\dagger\|\ge t  \right)
		\Prob_{R,v} \left( |u^*v|\ge \sqrt{1-s^2} \Bigm\vert \|R^\dagger\|\ge t \right),
		\]
		where we choose $v$ uniformly over the unit vectors on $\f C^n$.
		Observe that, by unitary invariance,
		\[
		\Prob_{R,v} \left( |u^*v|\ge \sqrt{1-s^2} \Bigm\vert \|R^\dagger\|\ge t \right) =
		\Prob_{v} \left(|e_1^*v|\ge \sqrt{1-s^2} \right) =
		\Prob_{v} \left( |v_1|\ge \sqrt{1-s^2} \right).
		\]
		The vector $v$ can be then seen as a unit real random vector with $2n$ entries, uniformly distributed on the unit sphere $S_{2n-1}$. From \eqref{eq:projection_beta},
		\[
		\Prob_{v} \left( |v_1|\ge \sqrt{1-s^2} \right)
		= (n-1) \int_{1-s^2}^1 (1-x)^{n-2} =
		(s^2)^{n-1}.
		\]
		Therefore,
		\[
		\Prob_{R,v} \left( \|R^\dagger v\|\ge t\sqrt{1-s^2}  \right)
		\ge 
		\Prob_{R} \left(  \|R^\dagger\|\ge t  \right) s^{2n-2}
		\]
		\begin{equation}\label{eq:pseudo1}
		\implies \Prob_{R} \left(  \|R^\dagger\|\ge t  \right)
		\le 
		\f P_{R,v} \left( \|R^\dagger v\|\ge t\sqrt{1-s^2}  \right) s^{2-2n}.
		\end{equation}
		Note now that
		\begin{equation}\label{eq:pseudo2}
		\f P_{R} \left( \|R^\dagger v\|\ge t\sqrt{1-s^2}  \right)
		=
		\f P_{R} \left( \|R^\dagger e_1\|\ge t\sqrt{1-s^2}  \right)
		\end{equation}
		since any unitary action on $R$ does not change its property to be the sum of a Gaussian matrix with mean 0 and variance $n^{-1} I$, plus a deterministic matrix. Moreover, $w = R^\dagger e_1$ is the first column of $R^\dagger$, and from $RR^\dagger = I$ (which is true with probability $1$ since $R$ is almost surely full rank), we know that $w^*$ is orthogonal to all the rows of $R$ but $r^T=e_1^T R$. Furthermore,
		\[
		r^T w = 1. 
		\]
		Let $r_\perp$ be the component of $r$ orthogonal to the vector space $\mathcal{V}$ generated by the other rows of $R$. 
		Since $RR^*$ is almost surely invertible, with probability $1$ we have $w^* = e_1^T (RR^*)^{-1} R$, so $w^*$ belongs to the row space of $R$. Since it is also orthogonal to $\mathcal{V}$, we conclude that $w^*$ and $r_\perp^T$ are parallel, and
		\[
		1 = r^Tw = r_\perp^Tw \implies 
		1 = \|r_\perp\| \cdot \|w\|.  
		\]
		Let $\vf$ be the density of the random matrix $R$, that can be split into $\vf = \psi\rho$ 
		where $\psi$ is the density of $r$, and $\rho$ is the density of the rest of the rows, say, $R_1$. We have
		\begin{equation}\label{eq:densities}
		\Prob_{R} \left( \|R^\dagger e_1\|\ge p^{-1}  \right)
		=
		\int_{\|w\|\ge p^{-1}} \vf(R) \, {{\rm d}}R
		=
		\int_{R_1} \rho \int_{\|r_\perp\|\le p} \psi \, {{\rm d}}r\, {{\rm d}}R_1
		\end{equation}
		Let us focus on the integral over $r$. Fix $R_1$ as a set of $n-1$ linearly independent (almost surely) vectors, with span $\mathcal{V}$. On the other hand, $r_\perp$ is the projection of $r$ over $\mathcal{V}^\perp$. If we split $r = r_N + r_D$ where $r_N$(the first row of $G$)  is a random Gaussian vector $N(0,n^{-1}I)$ and $r_D$ (the first row of $R_D$) is a deterministic vector of bounded norm, then we conclude
		\[
		UU^*r  = r_\perp = UU^*r_N + UU^* r_D.
		\]
		Here, $U$ is a $N\times N-n+1$ matrix such whose columns are an orthonormal basis of $\mathcal{V}^\perp$. We can rewrite $U=QE$ where $Q$ is unitary and $E^T= [I\,\, 0]$, so that 
		\[
		UU^* = Q
		\begin{pmatrix}
		I & 0\\
		0 & 0
		\end{pmatrix}
		Q^*.
		\]
		The vector $Q^* r_N$ is still  a random Gaussian vector with mean $0$ and variance $n^{-1}I$, and hence, 
		\[
		Q^*r_\perp =
		\begin{pmatrix}
		\wt r_N \\
		0
		\end{pmatrix}
		+
		\begin{pmatrix}
		\wt r_D \\
		0
		\end{pmatrix} 
		\]
		where $\wt r_N$ is a random Gaussian vector, of length $N-n+1$, with mean $0$ and variance $n^{-1}$, while $\|\wt r_D\|\le \|r_D\|$. Setting $\wt r_\perp =\wt r_N + \wt r_D$, then
		\[
		Q^* r = \begin{pmatrix}
		\wt r_\perp \\
		*
		\end{pmatrix} 
		\] 
		so
		\begin{equation}\label{eq:densities2}
		\int_{\| r_\perp\|\le p} \psi \, {{\rm d}}r
		=
		\int_{\|\wt r_\perp\|\le p} \wt\psi \, {{\rm d}}\wt r
		\end{equation}
		where $\wt \psi$ is the distribution of $\wt r_\perp$, and $\wt r$ is a generic vector of length $N-n+1$. In other words, we are restricting to the first $N-n+1$ coordinates of $r$, up to a unitary transformation. 
		Now, observe that, from \eqref{eq:densities} and \eqref{eq:densities2},
		\begin{equation}\label{eq:pseudo_e1}
		\Prob_{R} \left( \|R^\dagger e_1\|\ge p^{-1}  \right)
		=
		\int_{R_1}  
		\Prob_{v\sim \normal(v_\star,n^{-1}I)}\left(  \|v\| \le p    \right)\,
		\rho
		\,
		{{\rm d}}R_1
		\le 
		\int_{R_1}  
		\Prob_{v\sim \normal(0,2I)}\left(  \|v\|^2 \le 2 np^2    \right)\,
		\rho
		\,
		{{\rm d}}R_1
		\end{equation}
		where $v_\star$ is a generic complex vector of dimension $N-n+1$ and $v$ is a random complex vector of the same rank.
		Note further that a complex normally distributed vector $v$ of length $N-n+1$ and variance $2$ can be seen  as a real normally distributed vector $v_\R$ of length $2(N-n+1)$ and variance $1$. Taking this viewpoint, we write 
		\[
		\Prob_{v\sim \normal(0,2I)}\left(  \|v\| \le \sqrt{2 n}p    \right)
		=
		\Prob_{v_{\R}\sim \normal(0,I)}\left(  \|v_{\R}\|  \le \sqrt{2 n}p    \right)
		=
		\frac{1}{(2\pi)^{N-n+1}} 
		\int_{\|X\|\le \sqrt{2 n}p }
		e^{-\|X\|^2/2} 
		\, {{\rm d}}X
		\]
		\[
		\le 
		\frac{1}{(2\pi)^{N-n+1}} 
		\int_{\|X\|\le \sqrt{2 n}p }
		1 
		\, {{\rm d}}X
		=
		\frac{
			(\sqrt{2 n}p)^{2(N-n+1)}
			\pi^{N-n+1} 
		}{(2\pi)^{N-n+1}
			(N-n+1)!
		} 
		=
		\frac{(   np^2 )^{N-n+1}}{(N-n+1)!}.
		\]
		Recalling Stirling's bound
		\[ (N-n+1)! \ge  \sqrt{2 \pi (N-n+1)} \left( \frac{N-n+1}{e} \right)^{N-n+1}, \]
		we get the estimate
		\[
		\Prob_{v\sim N(0,2I)}\left(  \|v\| \le \sqrt{2 n}p    \right) \le 
		\frac 1{\sqrt{2\pi}}
		\frac{(np^2e)^{N-n+1}}{
			(N-n+1)^{N-n+3/2}
		}.
		\]
		The latter upper bound does not depend on $R_1$, so plugging it into \eqref{eq:pseudo_e1} we obtain
		\begin{equation}\label{eq:pseudo_e1_2}
		\Prob_{R} \left( \|R^\dagger e_1\|\ge p^{-1}  \right)
		\le 
		\int_{R_1}  
		\frac 1{\sqrt{2\pi}}
		\frac{(np^2e)^{N-n+1}}{
			(N-n+1)^{N-n+3/2}
		}
		\,
		\rho
		\,
		{{\rm d}}R_1
		=
		\frac 1{\sqrt{2\pi}}
		\frac{(np^2e)^{N-n+1}}{
			(N-n+1)^{N-n+3/2}
		}
		\end{equation}
		and, using \eqref{eq:pseudo1} and \eqref{eq:pseudo2},
		\[
		\Prob_{R} \left(  \|R^\dagger\|\ge t  \right)
		\le 
		\frac 1{\sqrt{2\pi}}
		\frac{\left(
			\frac{ne}{t^2}
			\right)^{N-n+1}}{
			(N-n+1)^{N-n+3/2}
		}
		\frac{1}{(1-s^2)^{N-n+1}(s^2)^{n-1}}
		\]
		for every $s\in(0,1)$ and $t>0$.
		We now sharpen the bound by optimizing in the parameter $s^2$: to this goal, we need to find a maximum over $(0,1)$ of
		\[
		q(Y) = (1-Y)^{N-n+1}Y^{n-1}.\]
		If $n>1$, then a straightforward computation shows that the maximum is achieved at
		\[ Y = \frac{n-1}{N} =: \lambda\in (0,1)\]
		and since
		\[
		\lambda^{-(n-1)} = \lambda^{-N\lambda} = (\lambda^{-\frac{\lambda}{1-\lambda}} )^{N(1-\lambda)}
		\le e^{N(1-\lambda)},
		\]
		we have that 
		\[
		\frac 1{q(\lambda)}
		=
		\frac{1}{(1-\lambda)^{N-n+1}\lambda^{(n-1)}}
		\le
		\frac{e^{N(1-\lambda)}}{(1-\lambda)^{N-n+1}}
		=
		\frac
		{(Ne)^{N-n+1}}
		{(N-n+1)^{N-n+1}}.
		\]
		On the other hand, if $n=1$, then $\sup_Y q(Y) = 1$ and 
		\[
		1\le 
		e^N
		=
		\frac
		{(Ne)^{N-n+1}}
		{(N-n+1)^{N-n+1}}.
		\]
		Thus,
		\[
		\Prob_{R} \left(  \|R^\dagger\|\ge t  \right)
		\le 
		\frac 1{\sqrt{2\pi} t^{2(N-n+1)} }
		\frac{\left(
			ne
			\right)^{N-n+1}}{
			(N-n+1)^{N-n+3/2}
		}
		\frac
		{(Ne)^{N-n+1}}
		{(N-n+1)^{N-n+1}},
		\]
		and hence, by taking $\tau = t^{-1}$,
		\[
		\Prob \left( \sigma_n(R)\le \tau  \right)
		\le 
		\frac { \tau^{2(N-n+1)}}{\sqrt{2\pi} }
		\frac{
			(nNe^2)^{N-n+1}
		}
		{
			(N-n+1)^{2(N-n+1) + 1/2}
		}.
		\]
		This concludes the proof.
	\end{proof}
	
	\subsection{Estimates on the singular values of certain random matrices in Section \ref{sec:n}}\label{sec:appn}

	In Lemma \ref{lem:smallestsingM} below, we obtain (in probability) a lower bound for the smallest singular value of the matrix $n^{-1/2}M - zI$, where $M$ is defined in in \eqref{eq:C}.
	\begin{lemma}\label{lem:smallestsingM}
		Let $M$ be the $kn \times kn$ matrix defined as in \eqref{eq:C} and $0 \neq z \in \C$. There exist constants $a,b >0$ such that, for every large enough $n$, 
		\[
		\Prob \left( {\sigma_{kn}(n^{-1/2}M - zI)}  <  n^{-a-2} \right) \le 2bn^{-2}
		\] 
		and in particular, with probability 1,
		\[
		{\sigma_{kn}(n^{-1/2}M - zI)}  \ge   n^{-a-2}
		\]
		for all large enough $n$.
	\end{lemma}
	\begin{proof}
		Using the same notation introduced in \eqref{eq:C}, let us rewrite 
		\[
		n^{-1/2}M - zI
		=n^{-1/2}E_1C^T + n^{-1/2}Z - zI
		\]
		and denote $N:=Z - n^{1/2} z I$. recall that the inverse of the least singular value of an invertible square matrix $X$ is equal to the spectral norm of $X^{-1}$.
		By Woodbury Lemma (Lemma \ref{lem:woodbury}), we see that
		\[
		(n^{-1/2}E_1C^T + n^{-1/2}Z - zI)^{-1} =
		n^{1/2} (N + E_1C^T)^{-1} =  
		n^{1/2}\left[ N^{-1} - N^{-1}E_1(  I + C^TN^{-1}E_1 )^{-1} C^T N^{-1} \right]
		\]
		Here we used that $(I + C^TN^{-1}E_1)$ is invertible with probability 1 and that $N$ is invertible since $z\ne 0$. As a consequence
		\[
		\frac 1{\sigma_{kn}(n^{-1/2}M - zI)} \le n^{1/2}\|N^{-1}\| \left( 1 + \|(  I + C^TN^{-1}E_1 )^{-1}\| \|C^T\| \|N^{-1}\|  \right).
		\]
		Observe now that $N = Z - n^{1/2}zI$ is a block Toeplitz matrix lower triangular matrix. It follows that its inverse  is also a block Toeplitz lower triangular matrix, and it is easily verified that the first block column of $N^{-1}$ is
		\[
		N^{-1} =
		\begin{pmatrix}
		(\sqrt n z)^{-1}I  & \\
		-(\sqrt n z)^{-2}I & \ddots \\
		\vdots & \ddots \\
		(-1)^{k-1}(\sqrt n z)^{-k}I & \ddots  
		\end{pmatrix}.
		\]
		Using $\|N^{-1}\|\le \sqrt{\|N^{-1}\|_1\|N^{-1}\|_{\infty}} =\|N^{-1}\|_1 $
		we can bound the norm from above with
		\[
		\|N^{-1}\| \le \|N^{-1}\|_1 = \sum_{i=1}^k (\sqrt n |z|)^{-i} = (\sqrt n |z|)^{-1} \frac{ (\sqrt n |z|)^{-k} -1 }{(\sqrt n |z|)^{-1} - 1 }
		=
		\frac{ 1-(\sqrt n |z|)^{-k}  }{\sqrt n |z| - 1 }
		\le \frac 2{\sqrt n |z|  }
		\]
		where we are assuming $ n \ge 4/|z|^2$.  Hence, 
		\[
		\frac 1{\sigma_{kn}(n^{-1/2}M - zI)} \le \frac 2{|z|} \left( 1 + \frac 2{\sqrt n |z|  }\|(  I + C^TN^{-1}E_1 )^{-1}\| \|C^T\|  \right).
		\]
		Note that $C^T$ is a $n\times kn$ matrix, and can be seen as a submatrix of a $kn\times kn$ random matrix $\wt C$ where every entry is an i.i.d copy of a Gaussian complex random variable $X$. Thanks to an interlacing theorem for singular values (Theorem \ref{cor:Interlacing}), we have that 
		$
		\| C^T\| \le \|\wt C\|
		$.
		On the other hand, by \autoref{thm:norm_bound}, 
		\begin{equation}\label{norm_C}
		\Prob (\| C^T\| > r\sqrt{n}) \le  
		\Prob (\| \wt C\| > r\sqrt{n}) \le
		s\exp(-crn)
		\end{equation}
		where $c,s,r>0$ are absolute constants. As a consequence, with high probability (at least $1 - s\exp(-crn) $), 
		\begin{equation}\label{eq:aux1}
		\frac 1{\sigma_{kn}(n^{-1/2}M - zI)} \le \frac 2{|z|} ( 1 + \frac {2r}{ |z|  }\|(  I + C^TN^{-1}E_1 )^{-1}\|   ).
		\end{equation}
		Consider now the matrix 
		\[
		C^TN^{-1}E_1 =  \sum_{i=1}^{k} (-1)^{i} (\sqrt n z)^{-i} C_{d-i}.
		\]
		Clear, each of its entry is a linear combination of i.i.d. Gaussian variables all having mean $0$, and this is still a normally distributed variable with mean $0$. Moreover, the variance is
		\[
		\sum_{i=1}^{k}  (\sqrt n |z|)^{-2i}
		=
		(\sqrt n |z|)^{-2}\frac{ (\sqrt n |z|)^{-2k} -1 }{(\sqrt n |z|)^{-2} - 1 }
		=
		\frac{1- (\sqrt n |z|)^{-2k}  }{(\sqrt n |z|)^{2} - 1 } =: \frac {c(n)^2}n = \Theta( \frac 1n).
		\] 
		Hence,
		\[
		\|(  I + C^TN^{-1}E_1 )^{-1}\| ^{-1}
		=
		\sigma_{n} (I+ C^TN^{-1}E_1)
		=
		\frac{c(n)}{\sqrt n}  \sigma_{n} \left(I \frac{\sqrt n}{c(n)} + G\right ) 
		\]
		where now $G$ is a matrix where all entries are i.i.d copies of a complex Gaussian random variable $X$ having mean $0$ and variance $1$. 
		Since for large values of $n$, $\|I \frac{\sqrt n}{c(n)}\| \le n$, we can apply \autoref{thm:least_sv_bound} and conclude that there exist positive constants $a,b$ such that 
		\[
		\Prob \left( \sigma_{n}\left(I \frac{\sqrt n}{c(n)} + G\right) \le n^{-a} \right)  \le bn^{-2}
		\] 
		meaning that, with high probability (at least $1-bn^{-2}$), it holds
		\begin{equation}\label{eq:aux2}
		\|(  I + C^TN^{-1}E_1 )^{-1}\| ^{-1} 
		\ge 
		\frac{c(n)}{\sqrt n}   n^{-a}  \ge n^{-a-1}
		\end{equation}
		for sufficiently large $n$. As a consequence, from \eqref{eq:aux1} and \eqref{eq:aux2},
		\[
		\frac 1{\sigma_{kn}(n^{-1/2}M - zI)} \le \frac 2{|z|} ( 1 + \frac {2r}{ |z|  }n^{a+1}   )
		\]
		and thus
		\[
		{\sigma_{kn}(n^{-1/2}M - zI)} \ge \frac 1{\frac 2{|z|} ( 1 + \frac {2r}{ |z|  }n^{a+1}   )} \ge n^{-a-2}
		\]
		with probability at least
		\[
		1 - bn^{-2} - s\exp(-crn)\ge 1 - 2bn^{-2}
		\]
		for any large enough values of $n$. In particular, we can conclude the proof by invoking the Borel-Cantelli lemma.
	\end{proof}
	
	Lemma \ref{lem:smallestsingE1CT} is the analogue of \ref{lem:smallestsingM} when the matrix $E_1C^T$ is considered.
	
	\begin{lemma}\label{lem:smallestsingE1CT}
		Let $E_1$ and $C^T$ be the matrices defined in \eqref{eq:C} and immediately after it. There exists a constant $\wt a>0$ such that, for all large enough $n$, 
		\[
		\Prob \left( {\sigma_{kn}(n^{-1/2}E_1C^T - zI)}  <  n^{-\wt a-2} \right) \le 2n^{-2}
		\] 
		and in particular, with probability $1$,
		\[
		{\sigma_{kn}(n^{-1/2}E_1C^T - zI)}  \ge   n^{-\wt a-2}
		\]
		for all sufficiently large $n$.
	\end{lemma}
	\begin{proof}
		The proof is very similar to Lemma \ref{lem:smallestsingM}, so we only sketch it. In this case, set $N := -n^{1/2}zI$ and
		\[
		(n^{-1/2}E_1C^T - zI)^{-1} =
		n^{1/2}\left[ N^{-1} - N^{-1}E_1(  I + C^TN^{-1}E_1 )^{-1} C^T N^{-1} \right]
		\]
		where $N$ is invertible since $z\ne 0$ and $ I + C^TN^{-1}E_1$ is almost surely invertible. We have $\|N^{-1}\| = n^{-1/2}/|z|$, and hence, with probability at least $1  - s\exp(-crn)$, 
		\[
		\frac 1{\sigma_{kn}(n^{-1/2}E_1C^T - zI)} \le \frac 1{|z|} ( 1 + \frac r{|z|  }\|(  I -\frac{1}{n^{1/2}z} C_{k-1} )^{-1}\|    ).
		\]
		Again, we write
		\[
		\|(  I -\frac{1}{n^{1/2}z} C_{k-1} )^{-1}\| ^{-1} 
		= 
		\sigma_n (I -\frac{1}{n^{1/2}z} C_{k-1}) 
		= 
		\frac{1}{n^{1/2}|z|}
		\sigma_n (-n^{1/2}zI + C_{k-1})
		\]
		Note that $\|-I n^{1/2}z\| \le n$ for $n$ big enough, so we can apply \autoref{thm:least_sv_bound} and find that with high probability (greater than $1-n^{-2}$), 
		\[
		\sigma_n (-n^{1/2}zI + C_{k-1}) \ge n^{-\wt a} 
		\]
		so that
		\[
		\frac 1{\sigma_{kn}(n^{-1/2}E_1C^T - zI)} \le \frac 1{|z|} ( 1 + r
		n^{\wt a+1} 
		)
		\]
		and
		\[
		\sigma_{kn}(n^{-1/2}E_1C^T - zI) \ge \frac {|z|}{ ( 1 + r
			n^{\wt a+1} )} \ge n^{-\wt a - 2}.
		\]
		By the Borel-Cantelli Lemma, the statement follows.
	\end{proof}
	
	In Lemma \ref{lem:spnAB}, we control the spectral norms of $M-zI$ and $E_1C^T-zI$.
	
	\begin{lemma}\label{lem:spnAB}
		Let $M,E_1,C^T$ be defined as in \eqref{eq:C} and immediately after it.	There exist constants $r,s>0$ such that for any $n$ large enough, 
		\[
		\Prob \left(
		\left\| n^{-1/2}E_1 C^T -zI\right \| >  r + |z| 
		\right)
		\le
		s\exp(-r^2n)
		\]
		\[
		\Prob \left(
		\left\| n^{-1/2}M -zI\right \| >  r + |z| + 1
		\right)
		\le
		s\exp(-r^2n)
		\]
		and in particular, there exists $d>0$ such that, with probability $1$,
		\[
		{\sigma_{1}(n^{-1/2}E_1 C^T -zI)}  \le   d,\qquad {\sigma_{1}(n^{-1/2}M -zI)}  \le   d
		\]
		for any sufficiently large $n$.
	\end{lemma}

	\begin{proof}
		Note that 
		\[
		\left\| n^{-1/2}M -zI\right \| \le \|n^{-1/2}E_1 C^T\| + |z| + n^{-1/2},
		\]
		\[
		\left\| n^{-1/2}E_1 C^T -zI\right \| \le \|n^{-1/2}E_1 C^T\| + |z|,
		\]
		and  $ \|E_1 C^T \|  =   \|C^T\|$, implying (see proof of Lemma \ref{lem:smallestsingM})
		\[
		\Prob (\| E_1 C^T\| > r\sqrt{n})  \le
		s\exp(-crn)
		\]
		where $c,s,r>0$ are absolute constants. The statement follows immediately.
	\end{proof}
	
	Lemma \ref{lem:sigmafnE1CT}, whose proof relies on Theorem \ref{thm:normpseudoinverse}, yields a probabilistic lower bound on the $f(n)$th singular value of $n^{-1/2} E_1 C^T - z I$.
	
	\begin{lemma}\label{lem:sigmafnE1CT}
		Let $0 < \delta < 1/2$,  $E_1, C^T$ be defined as in \eqref{eq:C} and immediately after it (so that $E_1 C^T$ is a $kn \times kn$ matrix with $k>1$), $f(n)=\lfloor kn-n^{1-\delta}\rfloor$, and $0 \neq z \in \C$. Then there exists a positive constant $t \leq 1$ and a positive constant $\ve > 0$ such that lower bound
		\[ \sigma_{f(n)} \left( n^{-1/2} E_1 C^T - z I \right) \ge tn^{\ve -1/2}   \]
		holds almost surely for all sufficiently large values of $n$.
	\end{lemma}
	
	\begin{proof}
		Note first that, denoting by $\wt T$ the matrix composed by the first $f(n)$ rows of $n^{-1/2}E_1 C^T  - zI$,  then by the interlacing theorem for singular values  (Theorem \ref{cor:Interlacing}) we have
		\[
		\sigma_{f(n)}(n^{-1/2}E_1 C^T  - zI)\ge \sigma_{f(n)}(\wt T)
		\]
		so it suffices to study $\wt T$. To this goal,  since $f(n)\ge n$, we partition
		\[
		\wt T =
		\begin{pmatrix}
		H -zI & L & P\\
		& zI & 
		\end{pmatrix}
		\]
		where $H$ is a $n\times n$ matrix, $L$ is a $n\times( f(n) - n  )$ matrix and $P$ is a $n\times (kn-f(n))$ matrix. Since  a permutation of the columns does not change the singular values, we can equivalently study the matrix
		\[
		T =
		\begin{pmatrix}
		L & H-zI & P\\
		zI & & 
		\end{pmatrix}
		=
		\begin{pmatrix}
		L & R\\
		zI &  
		\end{pmatrix}
		\]
		where $R := [H-zI\,\,\,  P ]$. $T$ is an $f(n)\times kn$ matrix, and its least singular value $\sigma_{\min}(T)$ satisfies
		\[
		\sigma_{\min}(T)^2 = \inf_{\|v\|= 1} \|v^* T\|^2 =
		\inf_{\|v\|= 1} \| v_1^* L + zv_2^*\|^2 + \|v_1^*R\|^2
		\ge 
		\inf_{\|v\|= 1} 
		(
		\| v_1^* L\| - \|zv_2^*\|
		)^2
		+ \|v_1^*R\|^2,
		\]
		where $v^* = [v_1^*\,  v_2^*]$. There are three possibilities. If $v_2$ is zero, then the minimum is attained as $\sigma_{\min} ( [H-zI\,\, L\,\, P])^2$ and if $v_1$ is zero, then the minimum is simply $|z|^2$.
		If neither is true, we can minimize the expression over $\|v_1\|\ne 0,1$. To this goal, denote $y = \|v_1\|^2$, so that $1-y= \|v_2\|^2$ and define $w_1=: v_1/\sqrt y$,   $w_2:= v_2/\sqrt{1-y}$. Then
		\[
		(
		\| v_1^* L\| - \|zv_2^*\|
		)^2
		+ \|v_1^*R\|^2 =
		(
		\sqrt y\| w_1^* L\| - \sqrt{1-y}|z|
		)^2
		+ y\|w_1^* R\|^2.
		\]
		Define now $\alpha := \| w_1^* L\|$, $\beta := |z|$, $\gamma^2 := \|w_1^*R\|^2$, so that we end up with the problem of minimizing the function
		\[
		g(y) = (
		\sqrt y\alpha - \sqrt{1-y}\beta
		)^2
		+ y\gamma^2
		= 
		\beta^2 + y(\gamma^2+\alpha^2-\beta^2) -2\alpha\beta\sqrt{y-y^2}.
		\]
		To further simplify the notation, it is convenient to introduce $a := \gamma^2+\alpha^2-\beta^2$ and $b:=2\alpha\beta$. This trick yields
		\[
		g(y) = \beta^2 + ya -b\sqrt{y-y^2};
		\]
		\[
		g'(y) = a +b\frac{2y-1}{2\sqrt{y-y^2}};
		\]
		\[
		g''(y) = \frac{b}{4} (y-y^2)^{-3/2}  \ge 0.
		\]
		In particular, the computation of the second derivative shows that $g(y)$ is a convex function, and hence, the roots of its derivative must correspond to minima. If $b>0$, then the minimum is also unique. Observe that by assumption $\beta\ne 0$, so $b=0$ implies $w_1^*L=0$ and  $\sigma_{\min}(T)^2$ is surely greater than $|z|^2$ or $\min_{\|w_1\|=1} \gamma^2$. Thus, we assume instead $b> 0$ and find a root of the derivative.
		\[
		g'(y) = 0\implies b(2y-1) = -2a\sqrt{y-y^2} \implies
		\mathrm{sign}(a) = \mathrm{sign}(1-2y)
		\]
		\[
		b^2(4y^2-4y+1) = 4a^2(y-y^2) \implies 
		4(a^2+b^2)(y^2-y+\frac 14) - a^2 = 0
		\implies
		(y-\frac 12)^2 = \frac{a^2}{4(a^2+b^2)}
		\]
		so the unique root of $g'(y)$ is
		\[
		y_\star = \frac 12 -\frac{a}{2\sqrt{a^2+b^2}} \in (0,1).
		\]
		Moreover,
		\[
		g(y_\star)= \beta^2 +\frac a2 - \frac{a^2}{2\sqrt{a^2+b^2}}
		- \frac{b^2}{2\sqrt{a^2+b^2}}
		= |z|^2 +\frac{a-\sqrt{a^2+b^2}}{2}
		\]
		\[
		= \frac{\gamma^2+\alpha^2+|z|^2-\sqrt{(\gamma^2+\alpha^2+|z|^2)^2 -4|z|^2\gamma^2}}{2}
		\ge 
		\frac {|z|^2}2
		\frac{\gamma^2}{\gamma^2+\alpha^2+|z|^2}
		\]
		To minimize the last expression, we can maximize the denominator by
		\[
		\gamma^2 + \alpha^2
		= \|w_1^* [L\, R]\|^2 \le \|[L\, R]\|^2
		\le 
		(
		\frac 1{\sqrt n} \|C^T\| + |z|
		)^2
		\]
		which is, with high probability (see \eqref{norm_C}), bounded by $(r+|z|)^2$. We can thus say that there exists a constant $t>0$ such that
		\[
		\sigma_{\min}(T)^2 \ge \min\{
		\min_{\|w_1\|=1} t^2\gamma^2,
		\sigma_{\min} ( [H-zI\,\, L\,\, P])^2,
		|z|^2
		\}.
		\]
		Yet, $ \min_{\|w_1\|=1} \|w_1^*R\| = \sigma_{\min}(R)$, and again by the interlacing theorem for singular values (Theorem \ref{cor:Interlacing}) $\sigma_{\min} ( [H-zI\,\, L\,\, P]) \ge \sigma_{\min}(R)$. We can therefore conclude that
		\begin{equation}\label{eq:interl}
		\sigma_{f(n)}(n^{-1/2}E_1 C^T   - zI)\ge\sigma_{\min}(T) \ge  \min\{t\sigma_{\min}(R), |z|\}
		\end{equation}
		for some absolute constant $0<t\le 1$ with high probability, and 
		almost surely for all $n$ sufficiently large.

		Almost surely, $R$ is full rank, and thus
		$ \sigma_{\min}(R) = \| R^\dagger \|^{-1}$. In particular, Theorem \ref{thm:normpseudoinverse}  yields the tail bound
		
		\[
		\Prob_{R} \left(  \sigma_n(R)\le \tau \right)
		\le 
		\left( 2\pi (n^{1-\delta}+1) \right)^{-1/2}
		\left( 
		\frac{k^2n^2 \tau^2 e^2}
		{	(n^{1-\delta}+1)^2}
		\right)
		^{\lceil n^{1-\delta}\rceil+1},
		\]
		where we used that $R$ is a $n\times (kn+n-f(n))$ matrix with both dimension less than $kn$, and
		\[
		n^{1-\delta}+1\le 
		\lceil n^{1-\delta}\rceil +1 = (kn+n-f(n))-n+1 .
		\]
		Now, fix any $\ve$ such that $0 < \ve < 1/2 -  \delta$ (as $0<\delta<1/2$, this is surely possible). Choosing $\tau = n^{\ve - 1/2}$, this implies that for $n$ big enough,
		\[
		\Prob_{R} \left(  \sigma_n(R)\le n^{\ve - 1/2} \right)
		\le \frac{1}{\sqrt{ 2\pi}} n^{\delta/2-1/2}
		\left( 
		\frac{(ke)^2 n^{1+2 \ve}}
		{	n^{2-2\delta}}
		\right)
		^{ n^{1-\delta}+1}.
		\]
		Setting $c:=1-2\delta-2\ve > 0$, we conclude that
		\[
		\Prob_{R} \left(  \sigma_n(R)\le n^{\ve - 1/2} \right)
		\le \frac{(ke)^2}{\sqrt{2 \pi}} n^{\delta/2-1/2-c}
		\left( 
		(ke)^{-2/c} n 
		\right)
		^{-c n^{1-\delta}},
		\]
		and the right hand side goes exponentially to zero. We conclude that $\sigma_{\min}(R)$ is at least of the order $n^{\ve -1/2}$ with high probability, and almost surely for sufficiently large $n$. From \eqref{eq:interl}, we get that for some absolute constant $0<t\le 1$ and some $\ve > 0$ then almost surely, for all $n$ sufficiently large,
		\[
		\sigma_{f(n)}(n^{-1/2}E_1 C^T  - zI) \ge  tn^{\ve -1/2}.
		\]
		
	\end{proof}

	\subsection{Estimates on the singular values of certain random matrices in Section \ref{sec:k}}\label{sec:appk}

	In Lemma \ref{lem:greatestsvA} below, we obtain probabilistic bounds for the first $n$ singular values of the matrix $M-zI$, where $M$ is defined in \eqref{eq:matrix_split_k}.
	\begin{lemma}\label{lem:greatestsvA}
		Let $M$ be the $kn \times kn$ matrix defined as in \eqref{eq:matrix_split_k} and $z \in \C$ with norm different from 1. There exist constants $r,s,c >0$ such that, for every $k>2$, 
		\[
		\Prob \left( \sigma_{1}(M-zI)  >  r\sqrt k + 1 +|z| \right) \le s\exp(-crk),
		\qquad 
		\sigma_n(M-zI)\ge|1-|z||.
		\] 
		In particular, with probability 1,
		\[
		\sigma_{1}(M-zI)  \le   r\sqrt k + 1 + |z|,
		\qquad 
		\sigma_n(M-zI)\ge|1-|z||
		\]
		for all large enough $k$.
	\end{lemma}
	\begin{proof}
		From the proof of Lemma \ref{lem:smallestsingM}, by exchanging the roles of $k$ and $n$, we know that 
		\begin{equation}\label{eq:norm_C}
		\Prob (\|  C^T\| > r\sqrt{k})  \le
		s\exp(-crk),
		\end{equation}
		where $r,s,c$ are absolute positive constants. 
		As a consequence, with high probability, 
		\[
		\sigma_1(M-zI)  \le\|C^T\| + 1 +  |z|\le  r\sqrt k + 1 +  |z|.
		\]
		Moreover, from \eqref{eq:bound_sv_B} and \eqref{eq:interlacing_AB}, we know that, if $k>2$,
		\[
		\sigma_n(M-zI)\ge  \sigma_{n+1}(M-zI)\ge \sigma_{2n+1}(B-zI) \ge |1-|z||>0.
		\]
	\end{proof}

	In Lemma \ref{lem:smallestsvA} below, we obtain a lower bound (valid with probability $1$) for the smallest singular value of the matrix $M-zI$, where $M$ is defined in \eqref{eq:matrix_split_k}.
	\begin{lemma}\label{lem:smallestsvA}
		Let $M$ be the $kn \times kn$ matrix defined as in \eqref{eq:matrix_split_k} and $z \in \C$ with norm different from 1. There exist a positive constant $t$ such that, with probability 1,
		\[ \sigma_{kn} \left( M - z I \right) \ge tk^{-2}  \]
		for all sufficiently large values of $k$.
	\end{lemma}
	\begin{proof}
		If we use Woodbury Lemma (Lemma \ref{lem:woodbury}) on the splitting 
		$
		M - zI = (B -zI) + E_1\wh C^T,
		$
		then\footnote{$B-zI$ is invertible, and here we are assuming $I + \wh C^T(B -zI)^{-1}E_1$ is also invertible, which is true with probability $1$. }
		\[
		(M-zI)^{-1} = (B -zI)^{-1} - (B -zI)^{-1}E_1(I + \wh C^T(B -zI)^{-1}E_1)^{-1}\wh C^T(B -zI)^{-1}
		\]
		and
		\begin{equation}\label{eq:norm_inverse_A}
		\|(M-zI)^{-1}\| \le
		\|(B-zI)^{-1}\|\left[ 
		1 +
		\|(B-zI)^{-1}\|
		\|\wh C^T\|
		\|(I + \wh C^T(B -zI)^{-1}E_1)^{-1}\|
		\right].
		\end{equation}
		$B-zI$ is a circulant matrix, so its inverse is still a circulant matrix and its norm can be estimated with \eqref{eq:bound_sv_B} by
		\begin{equation}\label{eq:norm_inverse_B}
		\|(B-zI)^{-1}\| = \sigma_{kn}(B-zI)^{-1} \le  \frac{1}{|1-|z||}.
		\end{equation}
		More specifically, $(B-zI)^{-1}$ is a block circulant matrix with first block column
		\[
		(B-zI)^{-1} =
		\begin{pmatrix}
		\frac{z^{k-1}}{1-z^k} I_n  & \\
		\frac{z^{k-2}}{1-z^k} I_n & \ddots \\
		\vdots & \ddots \\
		\frac{1}{1-z^k} I_n & \ddots  
		\end{pmatrix}.
		\]
		Consider now the matrix 
		\[
		I + \wh C^T(B -zI)^{-1}E_1 
		=
		\frac{z^k}{z^k-1}I   - \sum_{i=1}^k \frac{z^{k-i}}{1-z^k} C_{k-i}.
		\]
		It consists of the sum of a constant matrix, and a linear combination of i.i.d. Gaussian variables all having mean $0$. Such a linear combination is still a normally distributed variable with mean $0$ and variance 
		\[
		\sum_{i=1}^k \left| \frac{z^{k-i}}{1-z^k}\right|^2
		=
		\frac{1}{|1-z^k|^2}\sum_{i=0}^{k-1} |z|^{2i}
		=
		\frac{1}{1-|z|^2}
		\frac{1-|z|^{2k}}{|1-z^k|^2}:= c(k)^2
		\] 
		where 
		$c(k) \to |1-|z|^2|^{-1/2}$ for $k\to\infty$.
		We can thus write
		\[
		I + \wh C^T(B -zI)^{-1}E_1 
		=
		\frac{z^k}{z^k-1}I + c(k)G
		\]
		where $G$ is a Gaussian random $n\times n$ matrix where each entry has mean zero and unit variance.
		In Theorem \ref{thm:normpseudoinverse}, we proved that for any deterministic matrix $S$, we have  
		\begin{equation}\label{lsv}
		\f P \left(  \sigma_n(n^{-1/2}G +S)\le t \right)
		\le 
		\frac {t^2n^2e^2}{\sqrt{2\pi} 
		},
		\end{equation}
		so
		
		\begin{align*}
		\f P \left(   \|(I + \wh C^T(B -zI)^{-1}E_1)^{-1}\|\ge p \right)
		&=
		\f P \left( \sigma_n(I + \wh C^T(B -zI)^{-1}E_1 )\le \frac 1p \right)\\
		&=
		\f P \left(  \sigma_n
		\left( 
		\frac{z^k}{z^k-1}I + c(k)G
		\right) \le \frac 1p \right)\\
		&=
		\f P \left(  \sigma_n
		\left( 
		\frac{n^{-1/2}z^k}{c(k)(z^k-1)}I + n^{-1/2}G
		\right) \le (n^{1/2}c(k)p)^{-1} \right)\\
		&\le \frac {ne^2}{\sqrt{2\pi} c(k)^2p^2
		}.
		\end{align*}
		Since $ne^2/\sqrt{2\pi}$ does not depend on $k$, we can choose $p=k$ and conclude by the Borel-Cantelli Lemma that with probability 1 and for any $k$ sufficiently large,
		\begin{equation}\label{eq:norm_inverse}
		\|(I + \wh C^T(B -zI)^{-1}E_1)^{-1}\| < k.
		\end{equation}
		Finally, from \eqref{eq:norm_C}, we have that almost surely, for all $k$ sufficiently large,
		\begin{equation}\label{eq:norm_C_hat}
		\|\wh C^T\|\le \|C^T\| + 1 \le 1 + r\sqrt k
		\end{equation}
		for some absolute constant $r>0$. 
		Gathering all the bounds \eqref{eq:norm_inverse_B},  \eqref{eq:norm_inverse}, \eqref{eq:norm_C_hat} and substituting into \eqref{eq:norm_inverse_A},  we find that 
		\[
		\|(M-zI)^{-1}\| \le
		\frac{1}{|1-|z||}\left[ 
		1 +
		\frac{1}{|1-|z||}
		(r\sqrt k  +1)k
		\right]\le t^{-1}k^2
		\implies
		\sigma_{kn}(M-zI) \ge tk^{-2}
		\]
		for sufficiently large $k$ with probability 1, where $t$ is a positive constant depending only on $n$ and $z$.

	\end{proof}

\end{document}